\documentclass[a4paper,11pt,british]{scrreprt}
\usepackage{babel}
\usepackage[latin1]{inputenc}
\usepackage[T1]{fontenc}
\usepackage{lmodern}
\usepackage{amsmath}
\usepackage{amsthm}
\usepackage{mathtools} 
\usepackage{amssymb}
\usepackage{amsfonts}
\usepackage{latexsym}
\usepackage{enumerate}
\usepackage{enumitem}
\usepackage{hyperref}
\input xy
\xyoption{all}

\title{$\cs$-algebras associated to shift spaces}
\author{Notes for the summer school\\ \emph{Symbolic dynamics and
   homeomorphisms of the Cantor set},\\ University of Copenhagen, 23
 -- 27 June 2008}
\date{Toke Meier Carlsen\footnote{Toke.Carlsen@gmail.com}}

\newtheorem{lemma}{Lemma}[section]
\newtheorem{proposition}[lemma]{Proposition}
\newtheorem{theorem}[lemma]{Theorem}

\newtheorem{fact}[lemma]{Fact}

\theoremstyle{definition}
\newtheorem{definition}[lemma]{Definition}
\newtheorem{example}[lemma]{Example}

\newtheorem{remark}[lemma]{Remark}

\newcommand{\cs}{C^*}
\newcommand{\cha}[1]{1_{#1}}
\newcommand{\gaug}{\gamma}
\newcommand{\tgaug}{\widetilde{\gamma}}

\newcommand{\inv}{^{-1}}

\newcommand{\C}{\mathbb{C}}

\newcommand{\Z}{\mathbb{Z}}
\newcommand{\N}{\mathbb{N}}
\newcommand{\No}{{\N_{0}}}
\newcommand{\T}{\mathbb{T}}

\newcommand{\Hilbert}{\mathsf{H}}
\newcommand{\hilbert}{\Hilbert}
\newcommand{\B}{\mathcal{B}}
\newcommand{\E}{E}
\newcommand{\BL}{l^\infty}

\newcommand{\Oalg}[1][\OSS]{\mathcal{O}_{#1}}
\newcommand{\Dalg}[1][\OSS]{\mathcal{D}_{#1}}
\newcommand{\Falg}[1][\OSS]{\mathcal{F}_{#1}}
\newcommand{\Aalg}[1][\OSS]{\mathcal{A}_{#1}}
\newcommand{\Xalg}{\mathcal{X}}
\newcommand{\Yalg}{\mathcal{Y}}

\newcommand{\OSS}[1][]{{\mathsf{X}_{#1}}}
\newcommand{\ETSS}{\tilde{\TSS}}
\newcommand{\osh}{\sigma}
\newcommand{\tsh}{\tau}
\newcommand{\TSS}[1][]{\Lambda_{#1}}
\newcommand{\al}{\mathfrak a}
\newcommand{\alb}{\mathfrak b}
\newcommand{\alc}{\mathfrak c}
\newcommand{\alwords}{\al^*}
\newcommand{\als}{{\al '}^*}
\newcommand{\emptyword}{\epsilon}
\newcommand{\cyl}[2]{C(#1,#2)}
\newcommand{\Past}{{\mathcal P}}

\newcommand{\lang}{\mathsf{L}}
\newcommand{\tra}{\mathcal{L}}
\newcommand{\limm}{\underset{\longrightarrow}{\lim}}

\providecommand{\bnorm}[1]{\biggl\lVert#1\biggr\rVert}
\providecommand{\norm}[1]{\lVert#1\rVert}
\providecommand{\abs}[1]{\lvert#1\rvert}

\DeclareMathOperator{\aut}{Aut}

\DeclareMathOperator{\spa}{span}
\DeclareMathOperator{\id}{Id}

\setenumerate{labelindent=\parindent,leftmargin=*,label=(\arabic*)}

\begin{document}
\maketitle

\chapter{$C^*$-algebras associated to shift spaces}

I will in these note give an introduction to $C^*$-algebras associated
to shift spaces (also called subshifts). Notice that these notes
contains an appendix about 
$C^*$-algebras, Morita equivalence and $K$-theory of $C^*$-algebras.

$C^*$-algebras associated to shift spaces was introduced by Kengo
Matsumoto in \cite{MR1454478} as a generalization of Cuntz-Krieger
algebras (cf. \cite{MR561974}), and all the major results about them
are essentially do to him. $C^*$-algebras associated to shift spaces
have been studied by Matsumoto and his collaborators in
\cite{MR1645334,MR1921090,MR2029162,MR1636715,MR1646513,MR1691469,MR1688137,MR1716953,MR1774695,MR1764930,MR1869072,MR1852456,MR1637788},
and I have together with various collaborators contributed in
\cite{MR2339368,MR2085388,MR2253009,MR1978330,MR2380472,MR2117431,MR2091486,MR2360917}.

The approach I will take in these notes, is a little bit different
from Matsumoto's original approach. One notable difference is that I
will associate $C^*$-algebras to \emph{one-sided} shift spaces,
whereas Matsumoto associate $C^*$-algebras to \emph{two-sided} shift
spaces, but there are other differences as well. See \cite[Section
7]{MR2360917} for a discussion of the relationship between the
different $C^*$-algebras that have been associated to shift spaces.

\section{$\cs$-algebras of one-sided shift spaces}
Let $\al$ be a finite set endowed with the discrete topology. We will
call this set the \emph{alphabet} and its elements \emph{letters}. Let
$\al^\No$ be the infinite
product space $\prod_{n=0}^\infty \al$
endowed with the product topology. The transformation
$\osh$ on $\al^\No$ given by $$\bigl(\osh(x)\bigr)_i=x_{i+1},\ i\in
\No,$$ is called the \emph{(one-sided) shift}. Let $\OSS$ be a shift invariant
closed subset of $\al^\No$ (by shift invariant we mean that
$\osh(\OSS)\subseteq \OSS$, not necessarily
$\osh(\OSS)=\OSS$). The topological dynamical system
$(\OSS,\osh_{|\OSS})$ is called a \emph{one-sided shift space} (or a
\emph{one-sided subshift}). 

\begin{example}
If $\al$ is an alphabet, then $\al^\No$ itself is a shift space. We
call $\al^\No$ for the full one-sided $\al$-shift.
\end{example}

We will denote $\osh_{|\OSS}$ by $\osh_{\OSS}$ or $\osh$ for simplicity,
and on occasion the alphabet $\al$ by $\al_{\OSS}$.
We denote the $n$-fold composition of $\osh$ with itself by
$\osh^n$, and we denote the preimage of a set $X$ under $\osh^n$ by
$\osh^{-n}(X)$.

A finite sequence $u=(u_1,\ldots ,u_k)$ of elements $u_i\in
\al$ is called a finite \emph{word}. The \emph{length} of $u$ is $k$
and is denoted by $|u|$. For each $k\in \N$, we let $\al^k$ be the
set of all words with length $k$, and we let
$\lang^k(\OSS)$ be the set of all words with length $k$ appearing in
some $x\in \OSS$. We let $\lang^0(\OSS)=\al^0$ denote the set
$\{\emptyword\}$ consisting of the empty word 
$\emptyword$ which has length 0. We set $\lang_l(\OSS)=\bigcup_{k=0}^l
\lang^k(\OSS)$ and $\lang(\OSS)=\bigcup_{k=0}^\infty \lang^k(\OSS)$ and likewise
$\al_l=\bigcup_{k=0}^l\al^k$ and $\alwords=\bigcup_{k=0}^\infty \al^k$.
The set $\lang(\OSS)$ is called the
\emph{language} of $\OSS$. Note
that $\lang(\OSS)\subseteq \alwords$ for every shift space.

If $u\in \al^*$ with $|u|>0$, then we will
by $u_1$ denote the first letter (the leftmost) letter of $u$, by
$u_2$ the second letter of $u$, and so on till $u_{|u|}$ which denotes the last
(the rightmost) letter of $u$. Thus $u=u_1u_2\dotsm u_{|u|}$. 

We will often denote an element $x=(x_n)_{n\in\No}$ of $\al^\No$ by
\begin{equation*}
  x_0x_1\dotsm ,
\end{equation*}
and if $u\in\al^*$, then we will by $u x$ denote the sequence
\begin{equation*}
  u_1u_2\dotsm u_{|u|}x_0x_1\dotsm .
\end{equation*}
We will also often for a sequence $x$ belonging to either $\al^\No$
or $\al^\Z$ and for integers $k< l$ belonging to the appropriate
index set denote $x_kx_{k+1}\dotsm x_{l-1}$ by $x_{[k,l[}$ and regard
it as an element of $\al^*$. Similarly, $x_{[k,\infty[}$ will denote
the element
\begin{equation*}
  x_kx_{k+1}\dotsm
\end{equation*}
of $\al^\No$.

\begin{definition}
Let $\OSS$ be a one-sided shift space.
We let $\BL(\OSS)$ be the $\cs$-algebras of bounded functions on $\OSS$.
We define two maps $\alpha:\BL(\OSS)\to\BL(\OSS)$ and $\tra:\BL(\OSS)\to\BL(\OSS)$
by for $f\in\BL(\OSS)$ and $x\in\OSS$ letting
\begin{equation*}
  \alpha(f)(x)=f(\osh(x))\text{ and } \tra(f)(x)=
  \begin{cases}
    \frac{1}{\#\osh\inv(\{x\})}\sum_{y\in\osh\inv(\{x\})}f(y)&\text{if
    }x\in\osh(\OSS),\\ 
    0&\text{if }x\notin\osh(\OSS).	
  \end{cases}
\end{equation*}
\end{definition}

\begin{definition}
Let $\OSS$ be a one-sided shift space over the alphabet $\al$.
For every pair $u,v$ of words
in $\al^*$, we let $\cyl{u}{v}$ denote the
subset 
\begin{equation*}
  \{vx\in\OSS\mid x,ux\in\OSS\}
\end{equation*}
of $\OSS$ which consists of those elements which begins with a $v$ and
which satisfies that the element obtained by replacing the beginning $v$
with $u$ also is an element of $\OSS$. 

We let $\Dalg$ be the $\cs$-subalgebra of $\BL(\OSS)$
generated by $\{\cha{\cyl{u}{v}}\mid
u,v\in\al^*\}$. 
\end{definition}

\begin{proposition} \label{prop:dalg}
Let $\OSS$ be a one-sided shift space over the alphabet $\al$. Then we have:
\begin{enumerate}
\item $C(\OSS)\subseteq\Dalg$, 
\item $\Dalg$ is the closure of 
  \begin{equation*}
    \spa\left\{\prod_{i=1}^n\cha{\cyl{u_i}{v_i}}\mid
    u_1,\dots,u_n,v_1,\dots,v_n\in\alwords\right\},
  \end{equation*}
\item $\Dalg$ is closed under $\alpha$ and $\tra$ (i.e.,
  $f\in\Dalg\implies \alpha(f),\tra(f)\in\Dalg$), 
\item if $\Xalg$ is a $\cs$-subalgebra of $\BL(\OSS)$ that is closed
  under $\alpha$ and $\tra$ and contains $C(\OSS)$, then
  $\Dalg\subseteq \Xalg$. 
\end{enumerate}
\end{proposition}

\begin{proof}
(1): For $u\in\alwords$ we have that
$Z(u):=\cyl{\emptyword}{u}=\{ux\in\OSS\mid x\in\OSS\}$ is a clopen
subset of $\OSS$, and thus that $\cha{Z(u)}\in C(\OSS)$. Since
$\{Z(u)\mid u\in\alwords\}$ separates the points of $\OSS$, it follows
from the Stone-Weierstrass Theorem that the $\cs$-subalgebra of
$\BL(\OSS)$ generated by $\{\cha{Z(u)}\mid u\in\alwords\}$ is equal to
$C(X)$. Thus $C(X)\subseteq \Dalg$. 

(2): By definition $\Dalg$ is the smallest $C^*$-subalgebra of
$\BL(\OSS)$ which contains $\{\cha{\cyl{u}{v}}\mid
u,v\in\al^*\}$. It is not difficult to check that the closure of 
\begin{equation*}
  \spa\left\{\prod_{i=1}^n\cha{\cyl{u_i}{v_i}}\mid
  u_1,\dots,u_n,v_1,\dots,v_n\in\alwords\right\}.
\end{equation*}
satisfies this condition.

(3): Since $\alpha$ is a $*$-homomorphism, and $\tra$ is linear and
continuous, it is enough to prove that
$\alpha(\cha{\cyl{u}{v}})\in\Dalg$ for all $u,v\in\alwords$, and
that $\tra(\prod_{i=1}^n\cha{\cyl{u_i}{v_i}})\in\Dalg$ for all
$u_1,\dots,u_n,v_1,\dots,v_n\in\alwords$, so let us do that:

If $u,v\in\alwords$, then we have
\begin{equation*}
  \alpha(\cha{\cyl{u}{v}})=\sum_{a\in\al}\cha{\cyl{u}{va}}\in\Dalg.
\end{equation*}

If $A,B\subseteq\OSS$ such that $\cha{A},\cha{B}\in\Dalg$, then
$\cha{A\cup B}=\cha{A}+\cha{B}-\cha{A}\cha{B}\in\Dalg$. Thus
$\cha{\osh^n(\OSS)}=\cha{\bigcup_{u\in\al^n}\cyl{u}{\emptyword}}\in\Dalg$. It
follows that the function
$1-\cha{\osh(\OSS)}+\sum_{a\in\al}\cha{\cyl{a}{\emptyword}}$ also
belongs to $\Dalg$. Let us denote it by $h$. 
We have for $x\in\OSS$ that
\begin{equation*}
  h(x)=
  \begin{cases}
    \#\osh\inv(\{x\})&\text{if }x\in\osh(\OSS),\\
    1&\text{if }x\notin\osh(\OSS).
  \end{cases}
\end{equation*}
Thus $h$ in invertible, and it follows from Fact \ref{fact:inv} that
$h\inv\in\Dalg$. So the function $\cha{\osh(\OSS)}-1+h\inv$
belongs to $\Dalg$. Let us denote it by $d$. 
We have for $x\in\OSS$ that
\begin{equation*}
  d(x)=
  \begin{cases}
    \frac{1}{\#\osh^{-1}(\{x\})}&\text{if }x\in\osh(\OSS),\\
    0&\text{if }x\notin\osh(\OSS).
  \end{cases}
\end{equation*}
If $u_1,\dots,u_n,v_1,\dots,v_n\in\alwords$, then either
$\prod_{i=1}^n\cha{\cyl{u_i}{v_i}}=0$, all the $v_i$'s are equal to
the empty word, or all the non-empty $v_i$'s
begin with the same letter $a'$. In the first case
$\tra(\prod_{i=1}^n\cha{\cyl{u_i}{v_i}})=0$, in the second case we
have
\begin{equation*}
  \tra\left(\prod_{i=1}^n\cha{\cyl{u_i}{v_i}}\right)=
  d\left(\sum_{a\in\al}\prod_{i=1}^n\cha{\cyl{u_ia}{\emptyword}}\right) 
  \in\Dalg,
\end{equation*}
and in the third case we have
\begin{equation*}
  \tra\left(\prod_{i=1}^n\cha{\cyl{u_i}{v_i}}\right)=
  d\cha{\cyl{a'}{\emptyword}}
  \prod_{i\in I}\cha{\cyl{u_i}{(v_i)_2(v_i)_3\dots (v_i)_{|v_i|}}}
  \prod_{i\in I'}\cha{\cyl{u_ia}{\emptyword}}\in\Dalg
\end{equation*}
where $I=\{i\in\{1,2,\dots,n\}\mid v_i\ne \emptyword\}$ and 
$I'=\{i\in\{1,2,\dots,n\}\mid v_i= \emptyword\}$.

(4): Let $\Xalg$ be a $\cs$-subalgebra of $\BL(\OSS)$ that is closed
under $\alpha$ and $\tra$ and contains $C(\OSS)$. For $n\in\No$
let $g_n$ be the
function
\begin{equation*}
  1-\tra^n(1)+\sum_{u\in\al^n}\bigl(\tra^n(\cha{z(u)})\bigr)^2.
\end{equation*}
Then $g_n\in\Xalg$, and for every $x\in\OSS$ we have
\begin{equation*}
  g_n(x)=
  \begin{cases}
    \frac{1}{\#\osh^{-n}(\{x\})}&\text{if }x\in\osh^n(\OSS),\\
      1&\text{if }x\notin\osh^n(\OSS).
  \end{cases}
\end{equation*}
Thus, $g_n$ is invertible. It follows from Fact \ref{fact:inv} that
$g_n\inv$ and hence 
$f_n:=g_n\inv+\tra^n(1)-1$ belong to $\Xalg$. For every $x\in\OSS$ we have
$f_n(x)=\#\osh^{-n}(\{x\})$. Thus if $u,v\in\alwords$, then
$\cha{\cyl{u}{v}}=\cha{Z(v)}\alpha^{|v|}(f_{|u|}\tra^{|u|}(\cha{z(u)}))\in\Xalg$.
Since $\Dalg$ is generated by $\{\cha{\cyl{u}{v}}\mid
u,v\in\alwords\}$, it follows that $\Dalg\subseteq\Xalg$.
\end{proof}

\begin{remark} \label{remark:D}
Notice that  it
follows from \cite[Theorem 1]{MR0346134} that $C(\OSS)=\Dalg$ if and
only if $\OSS$ is of finite type.
\end{remark}

\begin{definition}
Let $\OSS$ be a one-sided shift space over the alphabet $\al$. For $w\in\alwords$ we let
$\lambda_w$ be the map from $\BL(\OSS)$ to $\BL(\OSS)$ given by
\begin{equation*}
\lambda_w(f)(x)=
\begin{cases}
f(wx)&\text{if }wx\in\OSS,\\
0&\text{if }wx\notin\OSS,
\end{cases}
\end{equation*}
for $f\in\BL(\OSS)$ and $x\in\OSS$. 
\end{definition}

\begin{lemma} \label{lemma:lambda}
Let $\OSS$ be a one-sided shift space over the alphabet $\al$ and let
$w\in\alwords$. Then $\lambda_w$ is a $*$-homomorphism and
$\lambda_w(\Dalg)\subseteq\Dalg$. 
\end{lemma}

\begin{proof}
It is easy to check that $\lambda_w$ is a $*$-homomorphism. Since $\Dalg$ is generated by $\{\cha{\cyl{u}{v}}\mid u,v\in\alwords\}$ and $\lambda_w$ is a $*$-homomorphism, it is enough to check that $\lambda_w(\cha{\cyl{u}{v}})\in\Dalg$ for all $u,v\in\alwords$, and this follows from the fact that
\begin{equation*}
\lambda_w(\cha{\cyl{u}{v}})=
\begin{cases}
\cha{\cyl{w}{\emptyword}}\cha{\cyl{uw'}{\emptyword}}&\text{if }w=vw',\\
\cha{\cyl{w}{\emptyword}}\cha{\cyl{u}{v'}}&\text{if }wv'=v,\\
0&\text{otherwise.}
\end{cases}
\end{equation*}
\end{proof}

\begin{definition} \label{def:rep}
Let $\OSS$ be a one-sided shift space over the alphabet $\al$. By a
\emph{representation} of $\OSS$ on a  
$\cs$-algebra $\Xalg$ we mean a pair 
$(\phi,(t_u)_{u\in\alwords})$ where $\phi$ is a $*$-homomorphism from
$\Dalg$ to $\Xalg$ 
and $(t_u)_{u\in\alwords}$ is a family of elements of $\Xalg$ such that
\begin{enumerate}
  \item $t_ut_v=t_{uv}$,  \label{item:1}
  \item $\phi(\cha{{\cyl{u}{v}}})=t_vt_u^*t_ut_v^*$ \label{item:2}
  \end{enumerate}
for all $u,v\in\al^*$.

We denote by $\cs((\phi,(t_u)_{u\in\alwords}))$ the $\cs$-subalgebra of $\Xalg$
generated by $\{t_u\mid u\in\alwords\}$.
\end{definition}

Let $\OSS$ be a one-sided shift space over the alphabet $\al$ and let
$\Hilbert$ be a Hilbert space 
with an orthonormal basis $(e_x)_{x\in \OSS}$ with the same
cardinality as $\OSS$ (we can for example let $\Hilbert$ be $l^2(\OSS)$ and
$e_x=\delta_x$). 

For every $u\in\alwords$, let $T_u$ be the bounded operator on
$\Hilbert$ defined by 
\begin{equation} \label{eq:3}
  T_u(e_x)=
  \begin{cases}
    e_{ux}&\text{if }ux\in\OSS,\\
    0&\text{if }ux\notin\OSS,
  \end{cases}
\end{equation}
and let $\phi:\Dalg\to\B(\Hilbert)$ be the $*$-homomorphism defined by
\begin{equation}
\phi(f)(e_x)=f(x)e_x. \label{eq:4}
\end{equation}

It is easy to check that $(\phi,(T_u)_{u\in\alwords})$ is a representation of $\OSS$. Thus we have:

\begin{proposition} \label{prop:representation}
  Let $\OSS$ be a one-sided shift space over the alphabet $\al$ and let
  $\Hilbert$ be a Hilbert space 
  with an orthonormal basis $(e_x)_{x\in \OSS}$ with the same
  cardinality as $\OSS$. For every $u\in\alwords$, let $T_u$ be the
  bounded operator on $\Hilbert$ defined by \eqref{eq:3}, and let
  $\phi:\Dalg\to\B(\Hilbert)$ be the $*$-homomorphism defined by
  \eqref{eq:4}. Then  
  $(\phi,(T_u)_{u\in\alwords})$ is a representation of $\OSS$ on $\B(\Hilbert)$.
\end{proposition}

\begin{theorem}[{cf. \cite[Remark 7.3]{MR2380472} and \cite[Theorem
    10]{MR2360917}}] \label{theorem:univ} 
 Let $\OSS$ be a one-sided shift space over the alphabet $\al$. There
 exists a $\cs$-algebra $\Oalg$ and a representation 
 $(\iota,(s_u)_{u\in\alwords})$ of $\OSS$ on $\Oalg$ satisfying:
\begin{enumerate}
\item $C^*(\iota,(s_u)_{u\in\alwords})=\Oalg$, 
\item if $(\phi,(t_u)_{u\in\alwords})$ is a representation of $\OSS$ on a $\cs$-algebra $\Xalg$, then there exists
a $*$-homomorphism $\psi_{(\phi,(t_u)_{u\in\alwords})}:\Oalg\to\Xalg$ such that $\psi_{(\phi,(t_u)_{u\in\alwords})}\circ\iota=\phi$ and $\psi_{(\phi,(t_u)_{u\in\alwords})}(s_u)=t_u$ for every $u\in\alwords$.
\end{enumerate}
\end{theorem}

The $\cs$-algebra $\Oalg$ can be constructed in different ways, for
example as the $\cs$-algebra of a groupoid (see \cite{phdthesis}), as
the $\cs$-algebra of a $\cs$-correspondence (see \cite{MR2380472}), or
as one of Ruy Exel's crossed product $\cs$-algebras of an endomorphism
and a transfer operator (see \cite{MR2360917}). 

We will through these notes let $(\iota,(s_u)_{u\in\alwords})$ denote the representation of $\OSS$ on $\Oalg$ mentioned in Theorem \ref{theorem:univ}.

\begin{remark}
We notice that since $\Oalg$ is generated by a countable family, it is
separable. 
\end{remark}

\begin{lemma}
Let $\OSS$ be a one-sided shift space over the alphabet $\al$. The $*$-homomorphism $\iota:\Dalg\to\Oalg$ is injective.
\end{lemma}

\begin{proof}
Notice that the $*$-homomorphism $\phi:\Dalg\to\B(\Hilbert)$ from Proposition 
\ref{prop:representation} is injective. It follows from Theorem \ref{theorem:univ} 
that there exists a $*$-homomorphism $\psi:\Oalg\to\B(\Hilbert)$ 
such that $\psi\circ\iota=\phi$. It follows that $\iota$ is injective.
\end{proof}

We will from now on view $\Dalg$ as a subalgebra of $\Oalg$ and
suppress $\iota$. 
This allows us to state and prove the following lemma about the fundamental
structure of $\Oalg$ which we will use throughout these notes without
any reference.

\begin{lemma} \label{lemma:structure}
  Let $\OSS$ be a one-sided shift space over the alphabet $\al$. We then have:
  \begin{enumerate}
  \item $s_\emptyword=s_\emptyword^*=s_\emptyword^2=\cha{\OSS}$ is a
    unit for $\Oalg$,
  \item if $u\in\alwords$, then $s_us_u^*=\cha{\cyl{\emptyword}{u}}$,
  \item if $u\in\alwords$, then $s_u$ is a partial isometry (i.e., $s_us_u^*s_u=s_u$ and $s_u^*s_us_u^*=s_u^*$),
  \item if $u,v\in\alwords$ and $|u|=|v|$, then we have
    \begin{equation*}
      s_u^*s_v=
      \begin{cases}
        \cha{\cyl{u}{\emptyword}}&\text{if }u=v,\\
        0&\text{if }u\ne v.
      \end{cases}
    \end{equation*}
  \end{enumerate}
\end{lemma}

\begin{proof}
  (1): Since $\emptyword\emptyword=\emptyword$, it follows from
  \ref{item:1} of Definition
  \ref{def:rep} that $s_\emptyword^2=s_\emptyword$
  from which it follows that
  $s_\emptyword s_\emptyword^*= (s_\emptyword
  s_\emptyword^*+(s_\emptyword-s_\emptyword s_\emptyword^*)) (s_\emptyword
  s_\emptyword^*+(s_\emptyword^*-s_\emptyword s_\emptyword^*))=
  s_\emptyword s_\emptyword^*+(s_\emptyword-s_\emptyword
  s_\emptyword^*)(s_\emptyword-s_\emptyword s_\emptyword^*)^*$. Thus
  $s_\emptyword=s_\emptyword s_\emptyword^*$ which shows that
  $s_\emptyword$ is self-adjoint and hence a projection. 

  Thus it follows from \ref{item:2} of Definition \ref{def:rep} that
  $s_\emptyword= s_\emptyword s_\emptyword^* s_\emptyword
  s_\emptyword^*= \cha{\cyl{\emptyword}{\emptyword}}= \cha{\OSS}$. It
  now immediately follows from \ref{item:1} of Definition
  \ref{def:rep} that $s_\emptyword$ is a unit for
  $\Oalg$. 

  (2): If $u\in\alwords$, then $s_u^*s_u=s_\emptyword s_u^*s_u
  s_\emptyword^*= \cha{\cyl{\emptyword}{u}}$, and $s_us_u^*=
  s_us_\emptyword^*s_\emptyword s_u^*=\cha{\cyl{u}{\emptyword}}$. 

  (3): Follows from (2).

  (4): Let $u,v\in\alwords$ with $|u|=|v|$. If $u\ne v$ then
  $\cyl{\emptyword}{u}\cap\cyl{\emptyword}{v}=\emptyword$ and so
  $s_u^*s_v= s_u^*s_us_u^*s_vs_v^*s_v=
  s_u^*\cha{\cyl{\emptyword}{u}}\cha{\cyl{\emptyword}{v}}s_v=0$.
\end{proof}

\begin{proposition} \label{prop:structure}
Let $\OSS$ be a one-sided shift space over the alphabet $\al$, let $n\in\No$, let $w\in\al^n$ and let $f\in\Dalg$. Then we have:
\begin{enumerate}
\item $\lambda_w(f)=s_w^*fs_w$, \label{item:3}
\item $s_w^*f=\lambda_w(f)s_w^*$, \label{item:4}
\item $\alpha^n(f)=\sum_{u\in\al^n}s_ufs_u^*$, \label{item:5}
\item $s_wf=\alpha^n(f)s_w$, \label{item:6}
\item $\sum_{u,v\in\al^n}s_us_v^*s_vs_u^*$ is equal to the function
  $x\mapsto \#\osh^{-n}(\{\osh^n(x)\})$ and is thus invertible, \label{item:7}
\item $\tra^n(f)=
  (\sum_{u\in\al^n}s_u)^*(\sum_{u,v\in\al^n}s_us_v^*s_vs_u^*)\inv f(\sum_{u\in\al^n}s_u)$. \label{item:8}
\end{enumerate}
\end{proposition}

\begin{proof}
  \ref{item:3}: It is clear that the map $f\mapsto s_w^*fs_w$ is linear and
  $*$-preserving. If $f,g\in\Dalg$, then it follows from Lemma 
  \ref{lemma:structure} and the fact that $\Dalg$ is commutative that we have
  \begin{equation*}
    s_w^*fs_ws_w^*gs_w=s_w^*f\cha{\cyl{\emptyword}{w}}gs_w=
    s_w^*\cha{\cyl{\emptyword}{w}}fgs_w=s_w^*s_ws_w^*fgs_w=s_w^*fgs_w,
  \end{equation*}
  which shows that the map $f\mapsto s_w^*fs_w$ is also
  multiplicative and thus is a $*$-homomorphism. According to 
  Lemma \ref{lemma:lambda}, $\lambda_w$ is a $*$-homomorphism, and
  since $\Dalg$ is generated by $\{\cha{\cyl{u}{v}}\mid
  u,v\in\alwords\}$, it is therefore enough
  to check that for $u,v\in\alwords$ we have
  $\lambda_w(\cha{\cyl{u}{v}})= s_w^*\cha{\cyl{u}{v}}s_w$, so let us do that:
  
  It is easy to check that 
  \begin{equation*}
    \lambda_w(\cha{\cyl{u}{v}})=
    \begin{cases}
      \cha{\cyl{w}{\emptyword}}\cha{\cyl{uw'}{\emptyword}}&\text{if }w=vw',\\
      \cha{\cyl{w}{\emptyword}}\cha{\cyl{u}{v'}}&\text{if }wv'=v,\\
      0&\text{otherwise.}
    \end{cases}
  \end{equation*}
  It follows from Lemma \ref{lemma:structure} that if $s_w^*s_v\ne 0$,
  then either $w=vw'$ for 
  some $w'\in\alwords$, or $v=wv'$ for some $v'\in\alwords$. In the
  first case we have  
  \begin{multline*}
    s_w^*s_vs_u^*s_us_v^*s_w=
    s_{w'}^*\cha{\cyl{v}{\emptyword}} \cha{\cyl{u}{\emptyword}}
    \cha{\cyl{v}{\emptyword}}s_{w'}=
    s_{w'}^*\cha{\cyl{v}{\emptyword}} \cha{\cyl{w'}{\emptyword}}
    \cha{\cyl{u}{\emptyword}}s_{w'}\\=
    s_{w'}^*s_v^*s_vs_{w'}^*s_{w'}s_u^*s_us_{w'}=
    s_{vw'}^*s_{vw'}s_{uw'}^*s_{uw'}= 
    \cha{\cyl{w}{\emptyword}}\cha{\cyl{uw'}{\emptyword}},
  \end{multline*}
  and the second case we
  have
  \begin{equation*}
    s_w^*s_vs_u^*s_us_v^*s_w=
    s_w^*s_ws_{v'}s_u^*s_us_{v'}^*s_w^*s_w= 
    \cha{\cyl{w}{\emptyword}}\cha{\cyl{u}{v'}}.
  \end{equation*}
  Thus
  $\lambda_w(\cha{\cyl{u}{v}})=s_w^*\cha{\cyl{u}{v}}s_w$ as wanted.

  \ref{item:4}: It follows from \ref{item:3}, Lemma
  \ref{lemma:structure} and the fact that $\Dalg$ is commutative that we have
  \begin{equation*}
  s_w^*f=s_w^*s_ws_w^*f=s_w^*\cha{\cyl{\emptyword}{w}}f=s_w^*f\cha{\cyl{\emptyword}{w}}=s_w^*fs_ws_w^*=\lambda_w(f)s_w^*.
  \end{equation*}

  \ref{item:5}: The map $f\mapsto \sum_{u\in\al^n}s_ufs_u^*$ is clearly linear
  and $*$-preserving. If $f,g\in\Dalg$, then it follows from Lemma
  \ref{lemma:structure} and the fact that $\Dalg$ is commutative that we have
  \begin{equation*}
    \left(\sum_{u\in\al^n}s_ufs_u^*\right)\left(\sum_{v\in\al^n}s_vgs_v^*\right)=
    \sum_{u\in\al^n}s_uf\cha{\cyl{u}{\emptyword}}gs_u^*= 
    \sum_{u\in\al^n}s_ufg\cha{\cyl{u}{\emptyword}}s_u^*= 
    \sum_{u\in\al^n}s_ufgs_u^*,
  \end{equation*}
  which proves that the map $f\mapsto \sum_{u\in\al^n}s_ufs_u^*$ is
  multiplicative, and thus a $*$-homomorphism. Since $\alpha^n$ is
  also a $*$-homomorphism and $\Dalg$ is generated by $\{\cha{\cyl{u}{v}}\mid
  u,v\in\alwords\}$, it is therefore enough
  to check that for $u',v'\in\alwords$ we have
  $\alpha^n(\cha{\cyl{u'}{v'}})=
  \sum_{u\in\al^n}s_u\cha{\cyl{u'}{v'}}s_u^*$, and that follows from
  the fact that
  \begin{multline*}
    \sum_{u\in\al^n}s_u\cha{\cyl{u'}{v'}}s_u^*=
    \sum_{u\in\al^n}s_us_{v'}s_{u'}^*s_{u'}s_{v'}^*s_u^*\\=  
    \sum_{u\in\al^n}s_{uv'}s_{u'}^*s_{u'}s_{uv'}^*=
    \sum_{u\in\al^n}\cha{\cyl{u'}{uv'}}=
    \alpha^n(\cha{\cyl{u'}{v'}}).
  \end{multline*}

  \ref{item:6}: It follows from \ref{item:5}, Lemma
  \ref{lemma:structure} and the fact that $\Dalg$ is commutative that we have
  \begin{equation*}
  s_wf=s_ws_w^*s_wf=s_w\cha{\cyl{w}{\emptyword}}f
  =s_wf\cha{\cyl{w}{\emptyword}}=s_wfs_w^*s_w
  =\sum_{u\in\al^n}s_ufs_u^*s_w=\alpha^n(f)s_w.
  \end{equation*}

  \ref{item:7}: Follows from Lemma \ref{lemma:structure} and \ref{item:5}.

  \ref{item:8}: If $u,v\in\al^n$ and $u\ne v$, then it follows from
  Lemma \ref{lemma:structure} and the fact that $\Dalg$ is
  commutative, 
  that we have
  \begin{multline*}
     s_u^*\left(\sum_{u,v\in\al^n}s_us_v^*s_vs_u^*\right)\inv f s_v
     = s_u^*\cha{\cyl{\emptyword}{u}}
     \left(\sum_{u,v\in\al^n}s_us_v^*s_vs_u^*\right)\inv f  
     \cha{\cyl{\emptyword}{v}} s_v\\
     = s_u^*\left(\sum_{u,v\in\al^n}s_us_v^*s_vs_u^*\right)\inv f 
     \cha{\cyl{\emptyword}{u}} \cha{\cyl{\emptyword}{v}} s_v=0. 		
  \end{multline*}
  Thus it follows from \ref{item:3} and \ref{item:7} that we have
  \begin{multline*}
    \left(\sum_{u\in\al^n}s_u\right)^*
    \left(\sum_{u,v\in\al^n}s_us_v^*s_vs_u^*\right)\inv 
    f\left(\sum_{u\in\al^n}s_u\right)
    = \sum_{w\in\al^n}
    \left(s_w^*\left(\sum_{u,v\in\al^n}s_us_v^*s_vs_u^*\right)\inv fs_w\right)\\
    =\sum_{w\in\al^n}\lambda_w
    \left(\sum_{u,v\in\al^n}s_us_v^*s_vs_u^*\right)\inv f)
    = \tra^n(f).
  \end{multline*}
\end{proof}

\begin{proposition} \label{prop:oalg}
Let $\OSS$ be a one-sided shift space over the alphabet $\al$. Then
$\Oalg$ is the closure of 
\begin{equation*}
\spa\{s_ufs_v^*\mid u,v\in\alwords,\ f\in\Dalg\}.
\end{equation*}
\end{proposition}

\begin{proof}
Let us by $\Xalg$ denote $\spa\{s_ufs_v^*\mid u,v\in\alwords,\ f\in\Dalg\}$. Since $\{s_u\mid u\in\alwords\}\subseteq \Xalg$, it suffices to prove that $\Xalg$ is a $*$-subalgebra of $\Oalg$. It is obvious that $\Xalg$ is closed under addition and conjugation, so it is enough to prove that if $u,v,u',v'\in\alwords$ and $f,f'\in\Dalg$, then $s_ufs_v^*s_{u'}f's_{v'}^*\in\Xalg$, so let us do that:

Let us first assume that $|v|\ge |u'|$. It follows from Lemma \ref{lemma:structure} that if $s_ufs_v^*s_{u'}f's_{v'}^*\ne 0$, then there exists a $w\in\alwords$ such that $v=u'w$, and in that case it follows from Proposition \ref{prop:structure} that we have
\begin{equation*}
s_ufs_v^*s_{u'}f's_{v'}^*=s_ufs_w^*s_{u'}^*s_{u'}f's_{v'}^*=s_ufs_w^*\cha{\cyl{u'}{\emptyword}}f's_{v'}^*=
s_uf\lambda_w(\cha{\cyl{u'}{\emptyword}}f')s_w^*s_{v'}^*\in\Xalg.
\end{equation*}

Let us then assume that $|v|\le |u'|$. It follows from Lemma \ref{lemma:structure} that if $s_ufs_v^*s_{u'}f's_{v'}^*\ne 0$, then there exists a $w\in\alwords$ such that $u'=vw$, and in that case it follows from Proposition \ref{prop:structure} that we have
\begin{equation*}
s_ufs_v^*s_{u'}f's_{v'}^*=s_ufs_v^*s_vs_wf's_{v'}^*=
s_uf\cha{\cyl{u'}{\emptyword}}s_wf's_{v'}^*= 
s_us_w\lambda_w(f\cha{\cyl{u'}{\emptyword}})f's_{v'}^*
\in\Xalg.
\end{equation*}
\end{proof}

An \emph{automorphism} of a $\cs$-algebra $\Xalg$ is a $*$-isomorphism from $\Xalg$ onto itself. We will by $\aut(\Xalg)$ denote the set of automorphisms of $\Xalg$. The set $\aut(\Xalg)$ becomes a group when equipt with composition. An \emph{action} of a group $G$ on a $\cs$-algebra $\Xalg$ is a homomorphism from $G$ to $\aut(\Xalg)$. We say that an action $\alpha:G\to \aut(\Xalg)$ of a topological group $G$ is \emph{strongly continuous} if for every convergent sequence $(g_n)_{n\in\N}$ in $G$ and every $x\in\Xalg$, the sequence $\alpha(g_n)(x)$ converges to $\alpha(\lim_{n\to\infty}g_n)(x)$.

We will by $\T$ denote the group $\{z\in\C\mid |z|=1\}$. The following lemma will be useful for checking if an action of $\T$ on a $\cs$-algebra is strongly continuous.

\begin{lemma} \label{lemma:strong}
  Let $\al$ be an alphabet. If $\Xalg$ is a $\cs$-algebra generated
  by a family $(x_u)_{u\in\alwords}$ and $\alpha:\T\to\Xalg$ is an
  action such that $\alpha(z)(x_u)=z^{|u|}x_u$ and for every $z\in\T$ and
  every $u\in\alwords$, then $\alpha$ is strongly continuous.
\end{lemma}

\begin{proof}
  Let $X$ be the set of elements $x$ of $\Xalg$ which satisfies that if
  $(z_n)_{n\in\N}$ converges to $z$ in $\T$, then $\alpha(z_n)(x)$
  converges to $\alpha_z(x)$ in $\Xalg$. It is straight forward to
  check that $X$ is a $\cs$-subalgebra of $\Xalg$, and since we for
  every $u\in\alwords$ have $x_u\in X$, it follows that $X=\Xalg$.
\end{proof}

\begin{proposition} \label{prop:gauge}
Let $\OSS$ be a one-sided shift space over the alphabet $\al$. Then
there exists a strongly continuous action $z\mapsto \gamma_z$ of $\T$
on $\Oalg$ such that $\gamma_z(s_u)=z^{|u|}s_u$ and $\gamma_z(f)=f$
for every $z\in\T$, $u\in\alwords$ and $f\in\Dalg$.
\end{proposition}

\begin{proof}
Let $z\in\T$. It is easy to check that $(\iota,(z^{|u|}s_u)_{u\in\alwords})$ is a representation of $\OSS$ on $\Oalg$.
Thus there exists a $*$-homomorphism $\gamma_z\Oalg\to\Oalg$ such that
$\gamma_z(s_u)=z^{|u|}s_u$ and $\gamma_z(f)=f$ for every
$u\in\alwords$ and every $f\in\Dalg$. 

If $z_1,z_2\in\T$ and $u\in\alwords$, then we have 
\begin{equation*}
\gamma_{z_1}\bigl(\gamma_{z_2}(s_u)\bigr)=\gamma_{z_1}(z_2^{|u|}s_u)=
z_1^{|u|}z_2^{|u|}s_u=(z_1z_2)^{|u|}s_u=\gamma_{z_1z_2}(s_u).
\end{equation*}
Since $\Oalg$ is generated by $\{s_u\mid u\in\alwords\}$, it follows
that $\gamma_{z_1}\circ\gamma_{z_2}=\gamma_{z_1z_2}$. We have in
particular that
$\gamma_{z}\circ\gamma_{z\inv}=\gamma_{z\inv}\circ\gamma_z=\gamma_{1}=\id_{\Oalg}$
for every $z\in \T$, so $\gamma_z\in\aut(\Oalg)$, and
$z\mapsto\gamma_z$ is an action of $\T$ on $\Oalg$. That this action
is strongly continuous follows from Lemma \ref{lemma:strong}. 
\end{proof}

The action of $\T$ on $\Oalg$ from Proposition \ref{prop:gauge} is
called \emph{the gauge action} of $\Oalg$. 
Since $\gaug$ is strongly continuous, it follows that we for every
$x\in\Oalg$ have that the function $z\mapsto \gaug_z(x)$ is a
continuous function from $\T$ to $\Oalg$. Thus we can make sense out
of the integral
\begin{equation*}
  \int_\T\gaug_z(x)dz
\end{equation*}
(cf. \cite[Lemma C.3.]{MR1634408}).

\begin{proposition} \label{prop:cond}
Let $\OSS$ be a one-sided shift space over the alphabet $\al$. If we
for every $x\in\Oalg$ let 
\begin{equation*}
  E(x)=\int_\T\gaug_z(x)dz,
\end{equation*}
then $E$ is a linear non-expensive contraction (i.e., $||E(x)||\le
||x||$ for all $x\in\Oalg$) from $\Oalg$ to itself such that
\begin{equation*}
E(s_ufs_v^*)=
\begin{cases}
  s_ufs_v^*&\text{if }|u|=|v|,\\
  0&\text{if }|u|\ne |v|
\end{cases}
\end{equation*}
for $u,v\in\alwords$ and $f\in\Dalg$.
\end{proposition}

\begin{proof}
It is clear that $E$ is linear. If $\phi:\T\to\Oalg$ is continuous,
then $\norm{\int_\T f(z)dz}\le \int_\T \norm{f(z)}dz$ (see \cite[Lemma
C.3.]{MR1634408}), and if 
$x\in\Oalg$, then 
$\norm{\gaug_z(x)}=\norm{x}$ for every $z\in\T$ since $\gaug_z$ is an
automorphism, so we have 
\begin{equation*}
  \norm{E(x)}=\bnorm{\int_\T\gaug_z(x)dz}\le \int_\T\norm{\gaug_z(x)}dz=
  \int_\T\norm{x}dz=\norm{x}.
\end{equation*}
Let $u,v\in\alwords$ and $f\in\Dalg$. If $|u|=|v|$, then
$\gaug_z(s_ufs_v^*)=s_ufs_v^*$ for every $z\in\T$, so
$E(s_ufs_v^*)=s_ufs_v^*$. If $|u|\ne |v|$, then we have for every
$z\in\T$ that $\gaug_z(s_ufs_v^*)=z^ns_ufs_v^*$ where $n=|u|-|v|\ne
0$, and since $\int_{\T}z^ndz=0$, it follows that $E(s_ufs_v^*)=0$.
\end{proof}

The map $E$ from Proposition \ref{prop:cond} is a so called 
\emph{faithful conditional expectation}. 

\begin{definition}
Let $\OSS$ be a one-sided shift space. We let $\Falg$ denote the
\emph{fix-point algebra} 
\begin{equation*}
\{x\in\Oalg\mid \forall z\in\T:\gamma_z(x)=x\}
\end{equation*}
of the gauge action $\gamma$ of $\Oalg$.
\end{definition}

Notice that $\Falg$ is a $\cs$-subalgebra of $\Oalg$.

\begin{proposition}
Let $\OSS$ be a one-sided shift space over the alphabet $\al$. Then we
have that $\Falg$ is equal to the closure of 
\begin{equation*}
\spa\{s_ufs_v^*\mid u,v\in\alwords,\ |u|=|v|,\ f\in\Dalg\},
\end{equation*}
and that $E(\Oalg)=\Falg$.
\end{proposition}

\begin{proof}
Let $\Xalg$ denote the closure of $\spa\{s_ufs_v^*\mid u,v\in\alwords,\
|u|=|v|,\ f\in\Dalg\}$. It is clear that $\Xalg\subseteq \Falg$, and
that $x=E(x)\in\ E(\Oalg)$ for every $x\in\Falg$. It follows from
Proposition \ref{prop:oalg} 
and \ref{prop:cond} that $E(\Oalg)=\Xalg$.
Thus we have $E(\Oalg)=\Xalg\subseteq\Falg\subseteq E(\Oalg)$ from
which the conclusion follows.
\end{proof}

The following theorem is an important and useful tool when one works
with $\Oalg$. 
I will not give a proof for it here.

\begin{theorem} \label{theorem:uniqueness}
  Let $\OSS$ be a one-sided shift space, $\Xalg$ a $\cs$-algebra
  and $\phi:\Oalg\to \Xalg$ a surjective $*$-homomorphism. Then the
  following two statements are 
  equivalent:
  \begin{enumerate}
  \item \label{item:u1} The $*$-homomorphism $\phi:\Oalg\to \Xalg$ is a
    $*$-isomorphism.
  \item \label{item:u2} The restriction of $\phi$ to $\Dalg$ is
    injective and there exists an action $\tgaug:\T\to\aut(\Xalg)$ such that
    $\tgaug_z(\phi(s_u))=z^{|u|}\phi(s_u)$
    for every $z\in\T$ and every $u\in\al^*$.
  \end{enumerate}
\end{theorem}
\section{One-sided conjugation}

\begin{definition}
Let $\OSS[1]$ and $\OSS[2]$ be one-sided shift spaces. We say that
$\OSS[1]$ and $\OSS[2]$ are \emph{conjugate} if there exists a
homeomorphism $\phi:\OSS[1]\to\OSS[2]$ such that
$\phi\circ\osh_{\OSS[1]}=\osh_{\OSS[2]}\circ\phi$. We call such a
homeomorphism for a \emph{conjugacy}.
\end{definition}

\begin{definition}
Let $\OSS$ be a one-sided shift space over the alphabet $\al$.
We will by $\lambda_{\OSS}$ denote the map
\begin{equation*}
  x\mapsto \left(\sum_{a\in\al}s_a^*\right)x\left(\sum_{b\in\al}s_b\right)
\end{equation*}
from $\Falg[\OSS]$ to $\Falg[\OSS]$.
\end{definition}

\begin{theorem} \label{theorem:osc}
Let $\OSS[1]$ and $\OSS[2]$ be one-sided shift spaces. If $\OSS[1]$
and $\OSS[2]$ are conjugate, then there exists a $*$-isomorphism $\psi$
from $\Oalg[{\OSS[1]}]$ to $\Oalg[{\OSS[2]}]$ such that
\begin{enumerate}
\item $\psi(C(\OSS[1]))=C(\OSS[2])$,
\item $\psi(\Dalg[{\OSS[1]}])=\Dalg[{\OSS[2]}]$,
\item $\psi(\Falg[{\OSS[1]}])=\Falg[{\OSS[2]}]$,
\item $\psi\circ\alpha=\alpha\circ\psi$,
\item $\psi\circ\tra=\tra\circ\psi$,
\item $\psi\circ\gaug_z=\gaug_z\circ\psi$ for every $z\in\T$,
\item $\psi\circ\lambda_{\OSS[1]}=\lambda_{\OSS[2]}\circ\psi$
\end{enumerate}
\end{theorem}

\begin{proof}
Let $\phi$ be a conjugacy between $\OSS[2]$ and $\OSS[1]$, and let $\Phi$
be the map between the bounded functions on $\OSS[1]$ and the bounded
functions on $\OSS[2]$ defined by
\begin{equation*}
    f\mapsto f\circ\phi.
\end{equation*}
Then $\Phi(C(\OSS[1]))=C(\OSS[2])$,
$\Phi\circ\alpha=\alpha\circ\Phi$ and
$\Phi\circ\tra=\tra\circ\Phi$, so it follows from Proposition \ref{prop:dalg} that
$\Phi(\Dalg[{\OSS[1]}])=\Dalg[{\OSS[2]}]$.

Let $\al_1$ be the alphabet of $\OSS[1]$ and $\al_2$ the alphabet of $\OSS[2]$. For $u\in\alwords_1$ and $v\in\alwords_2$ with $|u|=|v|$ let
$D(u,v)=\{x\in\OSS[2]\mid vx\in\OSS[2],\ \phi(vx)=u\phi(x)\}$ and $Z(u)=\cyl{\emptyword}{u}$. Then we have that $\cha{D(u,v)}=\lambda_v(\Phi(\cha{Z(u)})\in\Dalg[{\OSS[2]}]$.
For $u\in\alwords_1$ let $t_u=\sum_{v\in\al_2^{|u|}}s_v\cha{D(u,v)}$.

We will show that $(\Phi,(t_u)_{u\in\alwords_1})$ is a representation of $\OSS[1]$ on $\Oalg[{\OSS[2]}]$. 
If $u_1,u_2\in\alwords_1$ and $v_1,v_2\in\alwords_2$ with $|u_1|=|v_1|$ and $|u_2|=|v_2|$, then we have 
\begin{equation*}
\cha{D(u_1u_2,v_1v_2)}=\lambda_{v_2}\Bigl(\lambda_{v_1}\bigl(\Phi(\cha{Z(u_1)})\bigr)\Phi(\cha{Z(u_2)})\Bigr), 
\end{equation*}
so it follows from Proposition 
\ref{prop:structure} that we have
\begin{align*}
\cha{D(u_1,v_1)}s_{v_2}\cha{D(u_2,v_2)}&=
\lambda_{v_1}\bigl(\Phi(\cha{Z(u_1)})\bigr)s_{v_2}\lambda_{v_2}\bigl(\Phi(\cha{Z(u_2)})\bigr)\\
&=s_{v_2}\lambda_{v_2}\Bigl(\lambda_{v_1}\bigl(\Phi(\cha{Z(u_1)})\bigr)\Phi(\cha{Z(u_2)})\Bigr)\\
&=s_{v_2}\cha{D(u_1u_2,v_1v_2)}.
\end{align*}
It follows that if $u_1,u_2\in\alwords_1$, then we have
\begin{equation*}
t_{u_1}t_{u_2}=\sum_{v_1\in\al_2^{|u_1|}}
s_{v_1}\cha{D(u_1,v_1)}
\sum_{v_2\in\al_2^{|u_2|}}s_{v_2}\cha{D(u_2,v_2)}
=\sum_{v_1\in\al_2^{|u_1|}}\sum_{v_2\in\al_2^{|u_2|}}
s_{v_1}s_{v_2}\cha{D(u_1u_2,v_1v_2)}
=t_{u_1u_2}.
\end{equation*}
We also have that
\begin{align*}
t_{u_1}t_{u_2}^*t_{u_2}t_{u_1}^*
&=\sum_{v_1\in\al_2^{|u_1|}}s_{v_1}\cha{D(u_1,v_1)}
\sum_{v_2\in\al_2^{|u_2|}}\cha{D(u_2,v_2)}s_{v_2}^*
\sum_{v_3\in\al_2^{|u_2|}}s_{v_3}\cha{D(u_2,v_3)}
\sum_{v_4\in\al_2^{|u_1|}}\cha{D(u_1,v_4)}s_{v_4}^*\\
&=\sum_{v_1\in\al_2^{|u_1|}}\sum_{v_2\in\al_2^{|u_2|}}\sum_{v_4\in\al_2^{|u_1|}}
s_{v_1}\cha{D(u_1,v_1)}\cha{D(u_2,v_2)}\cha{\cyl{v_2}{\emptyword}}\cha{D(u_2,v_2)}\cha{D(u_1,v_4)}s_{v_4}^*\\
&=\sum_{v_1\in\al_2^{|u_1|}}\sum_{v_2\in\al_2^{|u_2|}}\sum_{v_4\in\al_2^{|u_1|}}
s_{v_1}\cha{D(u_1,v_1)}\cha{D(u_2,v_2)}\cha{D(u_1,v_4)}s_{v_4}^*\\
&=\sum_{v_1\in\al_2^{|u_1|}}\sum_{v_2\in\al_2^{|u_2|}}
s_{v_1}\cha{D(u_1,v_1)}\cha{D(u_2,v_2)}s_{v_1}^*\\
&=\sum_{v_1\in\al_2^{|u_1|}}\sum_{v_2\in\al_2^{|u_2|}}
\alpha^{|v_1|}(\cha{D(u_1,v_1)}\cha{D(u_2,v_2)})s_{v_1}s_{v_1}^*\\
&=\sum_{v_1\in\al_2^{|u_1|}}\sum_{v_2\in\al_2^{|u_2|}}
\alpha^{|v_1|}(\cha{D(u_1,v_1)}\cha{D(u_2,v_2)})\cha{Z(v_1)}\\
&=\Phi(\cha{\cyl{u_2}{u_1}}).
\end{align*}
Thus $(\Phi,(t_u)_{u\in\alwords_1})$ is a representation of $\OSS[1]$ on $\Oalg[{\OSS[2]}]$. It follows that there exists an isomorphism $\psi$ from $\Oalg[{\OSS[1]}]$ to $\Oalg[{\OSS[2]}]$ such that $\psi(s_u)=t_u=\sum_{v\in\al_2^{|u|}}s_v\cha{D(u,v)}$ for every $u\in\alwords_1$ and $\psi(f)=\Phi(f)$ for every $f\in\Dalg[{\OSS[1]}]$. 

One can in a similarly way prove that there exists an isomorphism $\eta$ from $\Oalg[{\OSS[2]}]$ to $\Oalg[{\OSS[1]}]$ such that $\psi(s_v)=\sum_{u\in\al_1^{|v|}}s_u\cha{\tilde{D}(v,u)}$ for every $v\in\alwords_2$ and $\rho(f)=\Phi^{-1}(f)$ for every $f\in\Dalg[{\OSS[2]}]$ where $\tilde{D}(v,u)=\{x\in\OSS[1]\mid ux\in\OSS[1],\ \phi^{-1}(ux)=v\phi\inv(x)\}$. If $u,u'\in\alwords_1$  with $|u|=|u'|$, then $\sum_{v\in\al_2^{|u|}}\cha{\tilde{D}(v,u')}\Phi\inv(\cha{D(u,v)})=0$ if $u\ne u'$, and $\sum_{v\in\al_2^{|u|}}\cha{\tilde{D}(v,u')}\Phi\inv(\cha{D(u,v)})=\cha{\cyl{u}{\emptyword}}$ if $u=u'$. It follows that if $u\in\alwords_1$, then we have
\begin{align*}
\rho(\psi(s_u))&=\rho(\sum_{v\in\al_2^{|u|}}s_v\cha{D(u,v)})=
\sum_{v\in\al_2^{|u|}}\sum_{u'\in\al_1^{|v|}}s_{u'}\cha{\tilde{D}(v,u')}\Phi\inv(\cha{D(u,v)})\\
&=s_u\cha{\cyl{u}{\emptyword}}=s_us_u^*s_u=s_u. 
\end{align*}
In a similar way, we can show that $\psi(\rho(s_v))=s_v$ for every $v\in\alwords_2$. Thus $\rho$ is the inverse of $\psi$, and $\psi$ is an isomorphism.

Since $\psi(f)=\Phi(f)$ for $f\in\Dalg[{\OSS[1]}]$, it follows that $\psi$ has the properties (1),(2),(4) and (5).
If $z\in\T$ and $u\in\alwords_1$, then we have
\begin{equation*}
\psi\bigl(\gamma_z(s_u)\bigr)=\psi(z^{|u|}s_u)=z^{|u|}\sum_{v\in\al_2^{|u|}}s_v\cha{D(u,v)}
=\gamma_z\left(\sum_{v\in\al_2^{|u|}}s_v\cha{D(u,v)}\right)=\gamma_z\left(\psi(s_u)\right).
\end{equation*}
It follows that $\psi$ has property (6). It follows from this and Proposition \ref{prop:gauge} that $\psi$ also has property (3).

Let $v\in\alwords_2$. Then we have that $\sum_{u\in\al_1^{|v|}}\cha{D(u,v)}=\cha{\cyl{v}{\emptyword}}$. Thus we have for every $x\in\Falg[{\OSS[1]}]$ that
\begin{align*}
\psi\bigl(\lambda_{\OSS[1]}(x)\bigr)&=\psi\Biggl(\biggl(\sum_{a\in\al_1}s_a^*\biggr)x\biggl(\sum_{b\in\al_1}s_b\biggr)\Biggr)
=\sum_{a\in\al_1}\sum_{b\in\al_1}\sum_{c\in\al_2}\sum_{d\in\al_2}\cha{D(a,c)}s_c^*\psi(x)s_d\cha{D(b,d)}\\
&=\sum_{c\in\al_2}\sum_{d\in\al_2}\cha{\cyl{c}{\emptyword}}s_c^*\psi(x)s_d\cha{\cyl{d}{\emptyword}}
=\sum_{c\in\al_2}\sum_{d\in\al_2}s_c^*\psi(x)s_d
=\lambda_{\OSS[2]}\bigl(\psi(x)\bigr).
\end{align*}
This proves that $\psi$ has property (7).
\end{proof}

\section{Two-sided conjugacy} \label{sec:two-sided-conjugacy}
Let $\al$ be a finite alphabet and let
$\al^\Z$ be the infinite
product space $\prod_{n\in\Z} \al$
endowed with the product topology. The transformation
$\tsh$ on $\al^\Z$ given by $$\bigl(\tsh(x)\bigr)_i=x_{i+1},\ i\in
\Z,$$ is called the \emph{two-sided shift}. Let $\TSS$ be a 
closed subset of $\al^\Z$ such that $\tsh(\TSS)=\TSS$. The topological dynamical system
$(\TSS,\tsh_{|\TSS})$ is called a \emph{two-sided shift space}. We will denote $\tsh_{|\TSS}$ by $\tsh_{\TSS}$ or just $\tsh$ for simplicity.

Given a two-sided shift space $\TSS$ we can construct a one-sided
shift space, namely 
\begin{equation*}
\{(x_n)_{n\in\No}\mid (x_n)_{n\in\Z}\in\TSS\}.
\end{equation*}
We will denote this one-sided shift space by $\OSS[{\TSS}]$.

Let $\TSS[1]$ and $\TSS[2]$ be two two-sided shift spaces. We say that
$\TSS[1]$ and $\TSS[2]$ are \emph{(topological) conjugate} if there exists a
homeomorphism $\psi:\TSS[1]\to\TSS[2]$ such that
$\psi\circ\tsh_{\TSS[1]}=\tsh_{\TSS[2]}$. I will in this section show that if
$\TSS[1]$ and $\TSS[2]$ are conjugate, then $\Oalg[{\OSS[{\TSS[1]}]}]$ and
$\Oalg[{\OSS[{\TSS[2]}]}]$ are Morita equivalent (cf. Section
\ref{sec:morita}) . This was proved in  
\cite{MR1764930} in the case where $\TSS[1]$ and $\TSS[2]$ both
satisfy two conditions called (I) and (E), and later in
\cite{MR2041268} under the assumption of (I), but we will here prove it
without any restrictions on $\TSS[1]$ and $\TSS[2]$. We will in our
proof closely follow the proof of \cite[Theorem 3.11]{MR2041268}, but
we will modify our proof such that it will work without the
requirement that $\TSS[1]$ and $\TSS[2]$ satisfy condition (I).

Like Matsumoto does in \cite{MR1764930}, we will use the notation of
\emph{bipartite code} introduced by Nasu (cf. \cite{MR857201}
and \cite{MR1234883}) who showed that every conjugacy between
two-sided shift spaces can be factorized into compositions of
bipartite codes. We will here briefly recall the necessary
definitions:

Let $\al$, $\al'_1$ and $\al'_2$ be alphabets. A one-to-one map from $\al$
to $\al'_1\al'_2:=\{bc\mid b\in\al'_1,\ c\in\al'_2\}$ is called a
\emph{bipartite expression} of $A$. Let $\TSS[1]$ and $\TSS[2]$ be two
two-sided shift spaces and let $f_1:\al_1\to\al'_1\al'_2$ be a bipartite
expression of the alphabet $\al_1$ of $\TSS[1]$. A map
$\phi:\TSS[1]\to\TSS[2]$ is called a \emph{bipartite code induced by}
$f_1$ if there exists bipartite expression $f_2:\al_2\to\al'_2\al'_1$ of the
alphabet $\al_2$ of $\TSS[2]$ such that either of the following
\ref{item:10} or \ref{item:11} is the case: 
\begin{enumerate}
\item If $(a_i)_{i\in\Z}\in\TSS[1]$,
  $\phi\bigl((a_i)_{i\in\Z}\bigr)=(d_i)_{i\in\Z}$ and $f_1(a_i)=b_ic_i$ with
  $b_i\in\al'_1$ and $c_i\in\al'_2$ for $i\in\Z$, then $f_2(d_i)=c_ib_{i+1}$
  for all $i\in\Z$. \label{item:10}
\item If $(a_i)_{i\in\Z}\in\TSS[2]$,
  $\phi\bigl((a_i)_{i\in\Z}\bigr)=(d_i)_{i\in\Z}$ and $f_1(a_i)=b_ic_d$ with
  $b_i\in\al_1$ and $c_i\in\al_2$ for $i\in\Z$, then $f_2(d_i)=c_{i-1}b_{i}$
  for all $i\in\Z$. \label{item:11}
\end{enumerate}

It is easy to check that a bipartite code is a conjugacy, and that if
$\phi$ is a bipartite code, then so is $\phi\inv$.

\begin{theorem}[Theorem 2.4 of \cite{MR857201}] \label{theorem:nasu}
  Any conjugacy between two-sided shift spaces can be decomposed into
  a composition of bipartite codes.
\end{theorem}

Thus, in order to prove that if two two-sided shift spaces $\TSS[1]$
and $\TSS[2]$ are conjugate, then the $C^*$-algebras
$\Oalg[{\OSS[{\TSS[1]}]}]$ and $\Oalg[{\OSS[{\TSS[2]}]}]$ are
Morita equivalent, it is enough to prove it in the case where there
exists a bipartite code between $\TSS[1]$ and $\TSS[2]$.

So, let $\TSS[1]$ and $\TSS[2]$ be two-sided shift spaces with
alphabets $\al_1$ and $\al_2$ respectively, and let $\alb$ and $\alc$
be two alphabets, $f_1:\al_1\to\al'_1\al'_2$ and $f_2:\al_2\to\al'_2\al'_1$ two
bipartite expression of $\al_1$ and $\al_2$ respectively, and let
$\phi:\TSS[1]\to\TSS[2]$ be a map such that condition \ref{item:10}
above holds. Let
$\al$ be the disjoint union of $\al'_1$ and $\al'_2$, and let $\tilde{f}_1$
and $\tilde{f}_2$ denote the maps from $\OSS[{\TSS[1]}]$ and $\OSS[{\TSS[2]}]$
respectively to $\al^\No$ given by
$\tilde{f}_1((a_i)_{i\in\No})=(f_1(a_i))_{i\in\No}$ and
$\tilde{f}_2((d_i)_{i\in\No})=(f_2(d_i))_{i\in\No}$ respectively. It is
easy to check that $\tilde{f}_1$ is a homeomorphism from
$\OSS[{\TSS[1]}]$ to $\tilde{f}_1(\OSS[{\TSS[1]}])$, that $\tilde{f}_2$ is
a homeomorphism from $\OSS[{\TSS[2]}]$ to
$\tilde{f}_2(\OSS[{\TSS[2]}])$, that $\tilde{f}_1(\OSS[{\TSS[1]}])$ and
$\tilde{f}_2(\OSS[{\TSS[2]}])$ are disjoint, and that
$\tilde{f}_1(\OSS[{\TSS[1]}])\cup \tilde{f}_2(\OSS[{\TSS[2]}])$ is a
closed and shift invariant subset of $\al^\No$ and thus a one-sided
shift space. We will denote the latter by $\OSS$.

Let $i\in\{1,2\}$. Denote the empty word of $\alwords_i$ by
$\emptyword_i$ and the empty word of $\alwords$ by $\emptyword$. We
extend $f_i$ to a map from $\alwords_i$ to $\alwords$ by setting
$f_i(u_1u_2\dotsm u_n)=f_i(u_1)f_i(u_2)\dotsm f_i(u_m)$ for
$u_1u_2\dotsm u_m\in\alwords_i\setminus\{\emptyword_i\}$, and
$f_i(\emptyword_i)=\emptyword$. 

Let $p_1$ be the characteristic function of
$\tilde{f}_1(\OSS[{\TSS[1]}])$ and let $p_2$ be the characteristic function of
$\tilde{f}_2(\OSS[{\TSS[2]}])$. We then have that $p_1,p_2\in\Dalg\subseteq\Oalg$
because
$p_1=\sum_{a\in\al'_1}\cha{\cyl{\emptyword}{a}}=\sum_{a\in\al'_1}s_as_a^*$
and
$p_2=\sum_{d\in\al'_2}\cha{\cyl{\emptyword}{d}}=\sum_{d\in\al'_2}s_ds_d^*$. For
$u\in\alwords_1\setminus\{\emptyword_1\}$, we let $t_u=s_{f(u)}$, and we let 
$t_{\emptyword_1}=p_1$. Likewise, for
$u\in\alwords_2\setminus\{\emptyword_2\}$, we let $t_u=s_{g(u)}$, and
we let $t_{\emptyword_2}=p_2$. 
\begin{lemma} \label{lemma:complementary}
  We have $p_1+p_2=1$.
\end{lemma}

\begin{proof}
  This follows from the fact that $1=\sum_{a\in\al}s_as_a^*$.
\end{proof}

\begin{lemma} \label{lemma:corners}
  Let $i\in\{1,2\}$. Then there
  exists a $*$-isomorphism from 
  $\Oalg[{\OSS[{\TSS[i]}]}]$ to $p_i\Oalg p_i$ 
  which for
  $u\in\alwords_i$ maps $s_u$ to $t_u$.
\end{lemma}

\begin{proof}
For $g\in\BL(\OSS[{\TSS[i]}])$ and $x\in\OSS$ let
\begin{equation*}
\phi_i(g)(x)=
\begin{cases}
g\bigl(\tilde{f}_i\inv(x)\bigr)&\text{if }x\in\tilde{f}_i(\OSS[{\TSS[i]}]),\\
0&\text{if }x\notin\tilde{f}_i(\OSS[{\TSS[i]}]).\\
\end{cases}
\end{equation*}
It is easy to check that $\phi$ then is an injective $*$-homomorphism
from $\BL(\OSS[{\TSS[i]}])$ to 
$\BL(\OSS)$. It is also easy to check that if $u,v\in\alwords_i$ are
not both the empty word, then
$\phi_i(\cha{\cyl{u}{v}})=\cha{\cyl{f_i(u)}{f_i(v)}}$, and that
$\phi_i(\cha{\cyl{\emptyword}{\emptyword}})=\cha{\tilde{f}_i(\OSS[{\TSS[i]}])}=p_i$. Thus 
$\phi_i$ maps 
$\Dalg[{\OSS[{\TSS[i]}]}]$ into $\Dalg$.
We will now show that $((t_u)_{u\in\alwords_1},\phi_i)$ is a
representation of $\OSS[{\TSS[i]}]$ on $\Oalg$. 

It is clear that if $u,v\in\alwords_i\setminus\{\emptyword\}$, then
$t_ut_v=t_{uv}$. 
  If $u,v\in\alwords_i$, and not both $u$ and $v$ are equal to
  $\emptyword_i$, then
  $\tilde{f}(\cyl{u}{v})=\cyl{f(u)}{f(v)}\subseteq
  \tilde{f}(\OSS[{\TSS[1]}])$. It follows that for
  $u\in\alwords_i\setminus\{\emptyword_1\}$ we have
  \begin{equation} \label{eq:1}
    t_{\emptyword_i}t_u=p_it_ut_u^*t_u=\cha{\tilde{f}(\OSS[{\TSS[i]}])}
    \cha{\cyl{\emptyword_i}{u}}t_u=\cha{\cyl{\emptyword_i}{u}}t_u=t_u
  \end{equation}
  and 
  \begin{equation} \label{eq:2}
    t_ut_{\emptyword_i}=t_ut_u^*t_up_i=t_u\cha{\cyl{u}{\emptyword_i}}
    \cha{\tilde{f}(\OSS[{\TSS[1]}])}= t_u\cha{\cyl{u}{\emptyword_i}}=t_u.
  \end{equation}
Thus $t_ut_v=t_{uv}$ for all $u,v_\in\alwords_i$.
  
  If $u,v\in\alwords_i$, and not both $u$ and $v$ are equal to
  $\emptyword_i$, then we have
  \begin{equation*}
    t_vt_u^*t_ut_v^*=s_{f(v)}s_{f(u)}^*s_{f(u)}s_{f(v)}^*=
    \cha{\cyl{f(u)}{f(v)}}=\phi_i(\cha{\cyl{u}{v}}). 
  \end{equation*}
  Since we also have that
  $t_{\emptyword_i}t_{\emptyword_i}^*t_{\emptyword_i}t_{\emptyword_i}^*
  =p_i=\phi_i(\cha{\cyl{\emptyword_i}{\emptyword_i}})$,
it follows that $((t_u)_{u\in\alwords_1},\phi_i)$ is a representation of $\OSS[{\TSS[i]}]$.
Thus there exists a $*$-homomorphism $\psi_i$ from
$\Oalg[{\OSS[{\TSS[i]}]}]$ to $\Oalg$. 

  The gauge action $\gaug$ of $\Oalg$ satisfies that
  $\gaug_z(t_u)=\gaug_z(s_{f(u)})=
  (z)^{|f(u)|}s_{f(u)}=z^{2|u|}t_u$ for
  $u\in\alwords_i\setminus\{\emptyword_i\}$, and since it also
  satisfies $\gaug_z(t_{\emptyword_i})=t_{\emptyword_i}$, we have that
  $\gaug_z$ leaves $\cs(t_u\mid u\in\alwords_i)$ invariant, and that
  if $z_1^2=z_2^2$ then $\gaug_{z_1}$ and $\gaug_{z_2}$ act equally on
  $\cs(t_u\mid u\in\alwords_i)$. Thus, if we for every $z\in\T$ let
  $\tgaug_z=\gaug_{z'|\cs(t_u\mid u\in\alwords_i)}$ where $z'\in\T$
  satisfies $z'^2=z$, then
  $\tgaug:\T\to\aut(\cs(t_u\mid u\in\alwords_i))$ is action which satisfies
  $\tgaug_z(t_u)=z^{|u|}t_u$ for every $z\in\T$. 
  Thus it follows from Theorem \ref{theorem:uniqueness} that $\psi$ is
  injective. 

  It follows from \eqref{eq:1} and \eqref{eq:2} that $\cs(t_u\mid
  u\in\alwords_i)\subseteq p_i\Oalg p_i$, and thus that 
  $\psi_i(\Oalg[{\OSS[{\TSS[i]}]}])\subseteq p_i\Oalg p_i$.

  Let $\Aalg$ be the $\cs$-subalgebra of $\Dalg$
  generated by $\{\cha{\cyl{u}{\emptyword}}\mid u\in\alwords\}$ and 
  $\Aalg[i]$ the $\cs$-subalgebra of $\Dalg[{\OSS[{\TSS[i]}]}]$
  generated by $\{\cha{\cyl{u}{\emptyword}}\mid u\in\alwords_i\}$. We
  then have that $\psi_i(\Aalg[i])\subseteq p_i\Aalg p_i$. We will now
  prove the inverse inclusion. Since every $f\in\Aalg$ commutes with $p_i$
  (because $p_i\in\Dalg$, $\Aalg\subseteq\Dalg$ and $\Dalg$ is
  commutative) and $\Aalg$ is generated by $\{s_v^*s_v\mid
  v\in\alwords\}$, it is enough to prove that
  $p_is_v^*s_vp_i\in\psi_i(\Aalg[i])$ for $v\in\alwords$. Assume that
  $v\ne\emptyword$ and $p_is_v^*s_vp_i\ne 0$. Since 
  $s_a=s_b=s_{ab}=0$ if
  $a,b\in\al'_1$ or $a,b\in\al'_2$, we have that $v\in
  f_i(\alwords_i)$ if $|v|$ is even, and that $v_1\in\al\setminus\al_i$ and
  $v_2v_3\dotsm v_{|v|}\in f_i(\alwords_i)$ if $|v|$ is uneven. In the
  former case, we have $s_v=t_u$ for some $u\in\alwords_1$ and thus that
  $p_is_v^*s_vp_i\in\psi_i(\Aalg[i])$. We may therefore assume that
  $|v|$ is uneven. If
  $a\in\al_i$, then $s_{av}=t_u$ for some $u\in\alwords_1$ and 
  $p_i\cha{\cyl{av}{\emptyword}}p_i=p_is_{av}^*s_{av}p_i\in\psi_i(\Aalg[i])$. If
  $A,B\subseteq\OSS$ and  
  $p_i\cha{A}p_i,p_i\cha{B}p_i\in\psi_i(\Aalg[i])$, then $p_i\cha{A\cup
    B}p_i\in \psi_i(\Aalg[i])$ because $p_i\cha{A\cup
    B}p_i=p_i\cha{A}p_i+p_i\cha{B}p_i-p_i\cha{A}p_ip_i\cha{B}p_i$. It
  follows that we have
  \begin{equation*}
    p_is_v^*s_vp_i=p_i\cha{\cyl{v}{\emptyword}}p_i=
    p_i\cha{\cup_{a\in\al_i}\cyl{av}{\emptyword}}p_i\in\psi_i(\Aalg[i]).
  \end{equation*}
  Thus $p_ifp_i\in\psi_i(\Aalg[i])$ for
  all $f\in\Aalg$.

  It is not difficult to show that $\spa\{s_vfs_v^*\mid
  v\in\alwords,\ f\in\Aalg\}$ is dense in $\Dalg$, so it follows from 
  Proposition \ref{prop:oalg} that $\spa\{s_vfs_w^*\mid v,w\in\alwords,\
  f\in\Aalg\}$ is dense in $\Oalg$. So in order to show that $p_i\Oalg
  p_i\subseteq \psi_i(\Oalg[{\OSS[{\TSS[i]}]}])$, it is therefore
  enough to prove that $p_is_vfs_w^*p_i\in \cs(t_u\mid u\in\alwords_i)$ for
  $u,v\in\alwords$ and $f\in\Aalg$. If $v\ne\emptyword$ and
  $p_is_v\ne 0$, then $v\in f_i(\alwords_i)$ 
  if $|v|$ is even, and $v_1v_2\dotsm v_{|v|-1}\in f_i(\alwords_i)$
  and $v_{|v|}\in\al'_{3-i}$ if $|v|$ is uneven. In the former case,
  we have $s_v\in\cs(t_u\mid u\in\alwords_i)$ and
  $p_is_v=s_v=s_vp_i$. In the latter case, we have $s_v=s_vp_{3-i}=
  s_v\sum_{a\in\al'_{3-i}}s_as_a^*=
  s_v\sum_{a\in\al'_{3-i}}s_ap_is_a^*$ and $s_vs_a\in\cs(t_u\mid
  u\in\alwords_i)$ for $a\in\al'_{3-i}$. Thus
  $p_is_vfs_w^*p_i=s_vp_ifp_is_w^*\in\cs(t_u\mid u\in\alwords_i)$ if
  both $|v|$ and $|w|$ are even. If both $|v|$ and $|w|$ are uneven,
  then
  $p_is_vfs_w^*p_i=\sum_{a,b\in\al'_{3-i}}p_is_vs_ap_is_a^*fs_bp_is_b^*s_w^*p_i$. We
  have $s_a^*fs_b=0$ if $a\ne b$, and $s_a^*fs_b\in\Aalg$ if
  $a=b$. Thus $p_is_vfs_w^*p_i\in\cs(t_u\mid u\in\alwords_i)$ in this
  case. If $|v|$ is uneven and $|w|$ is even then 
  \begin{equation*}
    p_is_vfs_w^*p_i=\sum_{a\in\al'_{1-i}}p_is_vs_ap_is_a^*fp_is_w^*=
      \sum_{a\in\al'_{3-i}}p_is_vs_ap_is_a^*fs_as_a^*p_is_w^*=0,
  \end{equation*}
  because $s_as_a^*p_i$ for $a\in\al'_{3-i}$. Likewise,
  $p_is_vfs_w^*p_i=0$ if $|v|$ is even and $|w|$ is uneven. Thus
  $p_is_vfs_w^*p_i\in\cs(t_u\mid u\in\alwords_i)$ for all
  $v,w\in\alwords$ and all $f\in\Aalg$, which proves that $p_i\Oalg
  p_i\subseteq \cs(t_u\mid u\in\alwords_i)$.
\end{proof}

\begin{lemma} \label{lemma:full}
  The projections $p_1$ and $p_2$ are full in $\Oalg$.
\end{lemma}

\begin{proof}
  We have
  $1=p_1+p_2=p_1+\sum_{a\in\al'_2}s_as_a^*=p_1+\sum_{a\in\al'_2}s_ap_1s_a^*$,
  which shows that the ideal generated by $p_1$ contains $1$, and thus
  that $p_1$ is a full projection in $\Oalg$. It follows from a
  similar argument that $p_2$ is a full projection in $\Oalg$.
\end{proof}

\begin{theorem} \label{prop:morita}
  Let $\TSS[1]$ and $\TSS[2]$ be two two-sided shift spaces for which
  there exists a bipartite code between $\TSS[1]$ and $\TSS[2]$. Then
  $\Oalg[{\OSS[{\TSS[1]}]}]$ and $\Oalg[{\OSS[{\TSS[2]}]}]$ are Morita
  equivalent.  
\end{theorem}

\begin{proof}
  Follows from Theorem \ref{theorem:nasu}, Lemma
  \ref{lemma:complementary}, Lemma \ref{lemma:corners} and Lemma
  \ref{lemma:full}. 
\end{proof}

\section{Flow equivalence}
We will in this section prove that if two two-sided shift spaces are
\emph{flow equivalent} (we refer to \cite{MR1418293,MR758893,MR0405385} and
\cite[Section 13.6]{MR1369092} for the definition of flow equivalence),
then the corresponding $\cs$-algebras are
stable isomorphic (and thus Morita equivalent by
\cite[Theorem 1.2.]{MR0463928}). This was first prove by Matsumoto in
\cite{MR1852456} under the assumption that both the two two-sided
shift spaces satisfy condition (I). We will closely follow Matsumoto's
original proof, but change it so it also work without the assumption
of condition (I).

Like Matsumoto, we will use the concept of \emph{symbolic
  expansion}. If $\al$ is an alphabet, then we let $\al'$ denote the
disjoint union of $\al$ and an extra symbol $*$ which does not belong
to $\al$. Choose a distinct
element $a_0$ of $\al$. For every $x\in\al^\Z$, we let $\eta(x)$ be
the element of $\al'^\Z$ obtained by replacing every occurrence of
$a_0$ in $x$ by $a_0*$ and letting all the other terms in $x$ be as
they are. It is easy to see that if $\TSS$ is a two-sided shift space over
$\al$, then $\{\eta(x)\mid x\in\TSS\}\cup\{\tsh(\eta(x))\mid
x\in\TSS\}$ is a two-sided shift space over the alphabet $\al'$. We call this
shift space for a \emph{symbolic expansion} of $\TSS$ and denote it by
$\ETSS$.

Parry and Sullivan proved in \cite{MR0405385} that flow equivalence
among two-sided shift spaces of finite type is generated by conjugacy and 
symbolic expansion. As noticed in \cite{MR1852456}, Parry and
Sullivan's proof also hold for two-sided shift spaces in
general. Thus we have:

\begin{theorem} \label{theorem:flow}
  Flow equivalence among two-sided shift spaces is generated by
  conjugacy and symbolic expansion. 
\end{theorem}

Thus, in order to prove that if two two-sided shift spaces are flow
equivalent, then the corresponding $\cs$-algebras are stable
isomorphic, it is enough to prove that if $\TSS$ is a two-sided shift
space, and $\ETSS$ is a symbolic expansion of $\TSS$, then
$\Oalg[{\OSS[{\TSS}]}]$ and $\Oalg[{\OSS[{\ETSS}]}]$ are stable
isomorphic.

Let $\TSS$, $\ETSS$, $\al$, $\al'$, $a_0$, $*$ and $\eta$ be as
above. For every $u\in\alwords$, we let $\eta(u)$ be
the element of $\al'^*$ obtained by replacing every occurrence of
$a_0$ in $u$ by $a_0*$ and letting all the other terms in $u$ be as
they are. We extend $\eta$ to a map from $\al^\No$ to
$\al'^\No$ in the obvious way. Let $\{\tilde{s}_u\mid u\in\al^*\}$ be
the canonical generators of 
$\Oalg[{\OSS[{\ETSS}]}]$. For $u\in\alwords\setminus\{\emptyword\}$,
we let $t_u=\tilde{s}_{\eta(u)}$, and we let
$t_\emptyword=\cha{\eta(\OSS[{\TSS}])}= 
\sum_{a\in\al}\cha{\cyl{\emptyword}{a}}=
\sum_{a\in\al}\tilde{s}_a\tilde{s}_a^*=1-\tilde{s}_*\tilde{s}_*^*$. 

\begin{lemma} \label{lemma:corner}
  Let $\{s_u\mid u\in\alwords\}$ be the family of canonical generators
  of $\Oalg[{\OSS[{\TSS}]}]$. Then there exists a $*$-isomorphism from
  $\Oalg[{\OSS[{\TSS}]}]$ to 
  $t_\emptyword\Oalg[{\OSS[{\ETSS}]}]t_\emptyword$ which for
  $u\in\alwords$ maps $s_u$ to $t_u$.
\end{lemma}

\begin{proof}
  For $f\in\BL(\OSS[{\TSS}])$ and $x\in\OSS[{\ETSS}]$ let 
  \begin{equation*}
    \phi(f)(x)=
    \begin{cases}
      f(y)&\text{if $x=\eta(y)$ for some $y\in\OSS[{\TSS}]$},\\
      0&\text{if }x\notin\eta(\OSS[{\TSS}]).
    \end{cases}
  \end{equation*}
  Then $\phi$ is a $*$-homomorphism from $\BL(\OSS[{\TSS}])$ to
  $\BL(\OSS[{\ETSS}])$. 
  If $u,v\in\alwords$ and $u$ and $v$ are not both equal to
  $\emptyword$, then we have
  \begin{equation*}
    \phi(\cha{\cyl{u}{v}})=\cha{\eta(\cyl{u}{v})}=\cha{\cyl{\eta(u)}{\eta(v)}}\in \Dalg[{\ETSS}]
  \end{equation*}
  Since we also have that
  $\phi(\cha{\cyl{\emptyword}{\emptyword}})=\phi(\cha{\OSS[{\TSS}]})=\cha{\eta(\OSS[{\TSS}])}=t_\emptyword\in\Dalg[{\OSS[{\ETSS}]}]$,
  it follows that $\phi$ maps $\Dalg[{\OSS[{\TSS}]}]$ into
  $\Dalg[{\OSS[{\ETSS}]}]$. We will now show that
  $((t_u)_{u\in\alwords},\phi)$ is a representation of $\OSS[{\TSS}]$
  on $\Oalg[{\OSS[{\ETSS}]}]$.
  
  If $u,v\in\alwords\setminus\{\emptyword\}$, then we have
  $t_ut_v=\tilde{s}_{\eta(u)}\tilde{s}_{\eta(v)}=\tilde{s}_{\eta(uv)}=t_{uv}$. If
  $u\in\alwords\setminus\{\emptyword\}$ then we have
  \begin{equation} \label{eq:5}
    t_\emptyword t_u=\cha{\eta(\TSS)}t_ut_u^*t_u=
    \cha{\eta(\TSS)}\cha{\cyl{\emptyword}{\eta(u)}}t_u=
    \cha{\cyl{\emptyword}{\eta(u)}}t_u= t_u 
  \end{equation}
  and
  \begin{equation} \label{eq:6}
    t_ut_\emptyword= t_ut_u^*t_u\cha{\eta(\TSS)}=
    t_u\cha{\cyl{\eta(u)}{\emptyword}}\cha{\eta(\TSS)}=
    t_u\cha{\cyl{\eta(u)}{\emptyword}}= t_u.
  \end{equation}
  Thus $t_ut_v=t_{uv}$ for all $u,v\in\alwords$.

  If $u,v\in\alwords$, and not both $u$ and $v$ are equal to
  $\emptyword$, then we have
  \begin{equation*}
    t_vt_u^*t_ut_v^*=\tilde{s}_{\eta(v)}\tilde{s}_{\eta(u)}^*
    \tilde{s}_{\eta(u)}\tilde{s}_{\eta(v)}^*=
    \cha{\cyl{\eta(u)}{\eta(v)}}=\cha{\eta(\cyl{u}{v})}=\phi(\cha{\cyl{u}{v}}). 
  \end{equation*}
  Since we also have that
  $t_{\emptyword}t_{\emptyword}^*t_{\emptyword}t_{\emptyword}^*
  =t_\emptyword=\cha{\eta(\OSS[{\TSS}])}=\phi(\cha{\cyl{\emptyword}{\emptyword}})$,
  we have that $((t_u)_{u\in\alwords},\phi)$ is a representation of $\OSS[{\TSS}]$
  on $\Oalg[{\OSS[{\ETSS}]}]$.
  Thus there exists a $*$-homomorphism $\psi:\Oalg[{\OSS[{\TSS}]}]\to
  \Oalg[{\OSS[{\ETSS}]}]$ such that $\psi(s_u)=t_u$ for every
  $u\in\alwords$, and such that the restriction of $\psi$ to
  $\Dalg[{\OSS[{\TSS}]}]$ is $\phi$ which is injective.

  It follows from the universal property of $\Oalg[{\OSS[{\ETSS}]}]$
  that there exists an action
  $\tgaug:\T\to\aut(\Oalg[{\OSS[{\ETSS}]}])$ such that
  $\tgaug_z(\tilde{s}_a)=z\tilde{s}_a$ for $z\in\T$ and $a\in\al$, and
  $\tgaug(\tilde{s}_*)=\tilde{s}_*$. We then have
  $\tgaug_z(t_u)=\tgaug_z(\tilde{s}_{\eta(u)})=z^{|u|}\tilde{s}_{\eta(u)}=
  z^{|u|}t_u$ for $u\in\alwords$ and $z\in\T$. Thus it follows from
  Theorem \ref{theorem:uniqueness} that $\psi$ is 
  injective.  

  It follows from equation \eqref{eq:5} and \eqref{eq:6} that $t_\emptyword
  t_ut_\emptyword=t_u$ for every $u\in\alwords$, so 
  $\psi(\Oalg[{\OSS[{\TSS}]}])=\cs(t_u\mid u\in\alwords)\subseteq t_\emptyword
  \Oalg[{\OSS[{\ETSS}]}]t_\emptyword$. 
  We will now prove that $t_\emptyword
  \Oalg[{\OSS[{\ETSS}]}]t_\emptyword\subseteq \cs(t_u\mid
  u\in\alwords)$. 
  Let $\Aalg[{\OSS[{\ETSS}]}]$ be the $\cs$-subalgebra of
  $\Dalg[{\OSS[{\ETSS}]}]$ 
  generated by $\{\cha{\cyl{u}{\emptyword}}\mid u\in\als\}$.
  It is not difficult to show that $\spa\{\tilde{s}_vf\tilde{s}_v^*\mid
  v\in{\alwords}',\ f\in\Aalg[{\OSS[{\ETSS}]}]\}$ is dense in
  $\Dalg[{\OSS[{\ETSS}]}]$, so it follows from  
  Proposition \ref{prop:oalg} that
  $\spa\{\tilde{s}_vf\tilde{s}_w^*\mid v,w\in\als,\
  f\in\Aalg[{\OSS[{\ETSS}]}]\}$ is dense in $\Oalg[{\OSS[{\ETSS}]}]$.
  It is therefore enough to prove that $t_\emptyword
  \tilde{s}_vf\tilde{s}_w^*t_\emptyword\in 
  \cs(t_u\mid u\in\alwords)$ for $v,w\in\als$ and 
  $f\in\Aalg[{\OSS[{\ETSS}]}]$. We will first prove this in the case
  where $v=w=\emptyword$. Since every $f\in\Aalg[{\OSS[{\ETSS}]}]$
  commutes with $t_\emptyword$ (because
  $t_\emptyword\in\Dalg[{\OSS[{\ETSS}]}]$,
  $\Aalg[{\OSS[{\ETSS}]}]\subseteq\Dalg[{\OSS[{\ETSS}]}]$ and
  $\Dalg[{\OSS[{\ETSS}]}]$ is 
  commutative) and $\Aalg[{\OSS[{\ETSS}]}]$ is generated by
  $\{\tilde{s}_v^*\tilde{s}_v\mid v\in\als\}$, it is enough to prove that
  $t_\emptyword \tilde{s}_v^*\tilde{s}_vt_\emptyword\in\cs(t_u\mid
  u\in\alwords)$ for $v\in\als$.

  Assume that $v\ne\emptyword$ and $t_\emptyword
  \tilde{s}_v^*\tilde{s}_vt_\emptyword\ne 0$. Since
  $\tilde{s}_{a_0}\tilde{s}_a=0$ for $a\in\al$, we 
  then have that $v_{|v|}\ne a_0$. So we either have that
  $v\in\eta(\alwords)$ or $v_1=*$. In the former case,
  $\tilde{s}_v\in\cs(t_u\mid u\in\alwords)$ and thus $t_\emptyword
  \tilde{s}_v^*\tilde{s}_vt_\emptyword\in\cs(t_u\mid u\in\alwords)$. In the latter
  case, we have $\cyl{v}{\emptyword}=\cyl{a_0v}{\emptyword}$ and
  $a_0v\in\eta(\alwords)$ from which it follows that 
  $s_v^*s_v=\cha{\cyl{v}{\emptyword}}=\cha{\cyl{a_0v}{\emptyword}}=
  \tilde{s}_{a_0v}^*\tilde{s}_{a_0v}\in\cs(t_u\mid u\in\alwords)$ and thus 
  $t_\emptyword \tilde{s}_v^*\tilde{s}_vt_\emptyword\in\cs(t_u\mid u\in\alwords)$.

  Let us now assume that $v\in\al'^*\setminus\{\emptyword\}$ and
  $t_\emptyword \tilde{s}_v\ne 0$. Since $t_\emptyword \tilde{s}_*=0$, we then have
  $v_1\in\al$. Thus, we either have that $v\in\eta(\alwords)$ or
  $v_{|v|}=a_0$. In the former case we have $\tilde{s}_v\in\cs(t_u\mid
  u\in\alwords)$ and $t_\emptyword \tilde{s}_v=\tilde{s}_v=\tilde{s}_v
  t_\emptyword$. In the latter case $t_\emptyword \tilde{s}_v=
  \tilde{s}_v=\tilde{s}_v\tilde{s}_*\tilde{s}_*^*=
  \tilde{s}_v\tilde{s}_*t_\emptyword \tilde{s}_*^*$ and
  $\tilde{s}_v\tilde{s}_*=\tilde{s}_{v*}\in\cs(t_u\mid u\in\alwords)$. 

  Thus if $v,w\in\al'^*$, $f\in\Aalg[{\OSS[{\ETSS}]}]$ and
  $t_\emptyword \tilde{s}_vf\tilde{s}_w^*t_\emptyword\ne 0$, then one
  of the following cases holds:
  \begin{enumerate}
  \item $t_\emptyword\tilde{s}_vf\tilde{s}_w^*t_\emptyword=
    \tilde{s}_vt_\emptyword f t_\emptyword \tilde{s}_w^*$ and
    $\tilde{s}_v,\tilde{s}_w\in\cs(t_u\mid u\in\alwords)$. 
  \item $t_\emptyword\tilde{s}_vf\tilde{s}_w^*t_\emptyword=
    \tilde{s}_v\tilde{s}_*t_\emptyword \tilde{s}_*^*f\tilde{s}_* t_\emptyword
    \tilde{s}_*^*\tilde{s}_w^*$ and
    $\tilde{s}_v\tilde{s}_*,\tilde{s}_w\tilde{s}_*\in\cs(t_u\mid u\in\alwords)$.
  \item $t_\emptyword \tilde{s}_vf\tilde{s}_w^*t_\emptyword=
    \tilde{s}_v\tilde{s}_*\tilde{s}_*^*f 
    t_\emptyword \tilde{s}_w^*$. 
  \item $t_\emptyword \tilde{s}_vf\tilde{s}_w^*t_\emptyword=
    \tilde{s}_vt_\emptyword f \tilde{s}_*\tilde{s}_*^*\tilde{s}_w^*$. 
  \end{enumerate}
  In the first case $t_\emptyword \tilde{s}_vf\tilde{s}_w^*t_\emptyword\in
  \cs(t_u\mid u\in\alwords)$ because $t_\emptyword ft_\emptyword\in
  \cs(t_u\mid u\in\alwords)$ as shown above. In the second case,
  $t_\emptyword \tilde{s}_vf\tilde{s}_w^*t_\emptyword\in \cs(t_u\mid u\in\alwords)$
  because $\tilde{s}_*^*f\tilde{s}_*\in\Aalg[{\OSS[{\ETSS}]}]$ and
  thus $t_\emptyword\tilde{s}_*^*f\tilde{s}_*t_\emptyword\in
  \cs(t_u\mid u\in\alwords)$ as shown above. The third and forth cases
  can actually not happen, because since $\tilde{s}_*\tilde{s}_*^*$
  commutes with every element of $\Aalg[{\OSS[{\ETSS}]}]$ (because
  $\tilde{s}_*\tilde{s}_*^*\in\Dalg[{\OSS[{\ETSS}]}]$,
  $\Aalg[{\OSS[{\ETSS}]}]\subseteq\Dalg[{\OSS[{\ETSS}]}]$ and
  $\Dalg[{\OSS[{\ETSS}]}]$ is commutative) and $t_\emptyword
  \tilde{s}_*=0$, we have that $\tilde{s}_*\tilde{s}_*^*f
  t_\emptyword= t_\emptyword f \tilde{s}_*\tilde{s}_*^*=0$.

  Hence $t_\emptyword \tilde{s}_vf\tilde{s}_w^*t_\emptyword\in
  \cs(t_u\mid u\in\alwords)$ for $u,v\in\al'^*$ and
  $f\in\Aalg[{\OSS[{\ETSS}]}]$, and we thus have $t_\emptyword
  \Oalg[{\OSS[{\ETSS}]}]t_\emptyword\subseteq \cs(t_u\mid u\in\alwords)$.
\end{proof}

\begin{lemma} \label{lemma:full2}
  The projection $t_\emptyword$ is full in $\Oalg[{\OSS[{\ETSS}]}]$.
\end{lemma}

\begin{proof}
  We have that
  $1=\sum_{a\in\al'}\tilde{s}_a\tilde{s}_a^*=
  t_\emptyword+\tilde{s}_*\tilde{s}_*^*=
  t_\emptyword+\tilde{s}_*t_\emptyword \tilde{s}_*^*$ from which the
  conclusion follows.
\end{proof}



\begin{theorem} \label{theorem:flowmorita}
  Let $\TSS[1]$ and $\TSS[2]$ be two two-sided shift spaces which are
  flow equivalent. Then $\Oalg[{\OSS[{\TSS[1]}]}]$ and
  $\Oalg[{\OSS[{\TSS[2]}]}]$ are Morita equivalent.
\end{theorem}

\begin{proof}
  This follows from Theorem \ref{theorem:flow} and Lemma
  \ref{lemma:corner} and \ref{lemma:full2}.
\end{proof}

\section{The $K$-theory of $C^*$-algebras associated to shift spaces }

Since $K_0(\Xalg)$ and $K_1(\Xalg)$ are invariants of a $\cs$-algebra
$X$, it follows 
from the previous section that $K_0(\Oalg)$, $K_1(\Oalg)$ and
$K_0(\Falg)$ are invariants of $\OSS$. In this section, we will
present formulas based on $l$-past equivalence for these
invariants. This was done in \cite{MR1646513,MR1691469,MR1852456}
by Matsumoto for the case of
one-sided shift spaces of the form $\OSS_\Lambda$, where $\Lambda$ is
a two-sided shift space and generalized to the general case in
\cite{speciale} (see also \cite{MR2360917}).
I will not here prove the formulas for $K_0(\Oalg)$, $K_1(\Oalg)$ and
$K_0(\Falg)$, because that would require a knowledge about $K$-theory
for $C^*$-algebras that I do not expect the reader to have, but only
establish the necessary setup and state the 
theorems which give the formulas. The interested reader can find
proofs of these theorems in the above mentioned references.

From these formulas, one can directly prove that $K_0(\Oalg)$, $K_1(\Oalg)$ and
$K_0(\Falg)$ are
invariants of $\OSS$ without involving $\cs$-algebras. This is done
(for one-sided shift spaces of the form $\OSS_\Lambda$, where
$\Lambda$ is a two-sided shift space) in
Matsumoto's very interesting paper \cite{MR1710375}, where also other
invariants of shift spaces are presented.

Let $\OSS$ be a one-sided shift space. 
We will for each $l\in\No$ define an equivalence relation on $\OSS$
called \emph{$l$-past equivalence}. These equivalence relations were
introduced by Matsumoto in \cite{MR1691469}. For $k\in\No$ and $x\in\OSS$
let $\Past_k(x)=\{u\in\al^k\mid ux\in\OSS\}$. 
If $x,y\in\OSS$ and $l\in\No$, then we say that \emph{$x$ and $y$ are
  $l$-past equivalent} and write $x\sim_l y$ if
$\bigcup_{k=0}^l\Past_k(x)=\bigcup_{k=0}^l\Past_k(y)$. Notice that
since $\al^k$ is finite for each $k\in\No$, we have for each
$l\in\No$ only finitely many $l$-past equivalence classes. We let
$m(l)$ be this number of $l$-past equivalence classes, and we denote the
$l$-past equivalence classes by
$\E_1^l,\E_2^l,\dotsc,\E_{m(l)}^l$.
For each  $l\in \No,\ j\in \{1,2,\dots ,m(l)\}$ and $i\in \{1,2,\dots
,m(l+1)\}$, let
\[ I_l(i,j)=\left\{ \begin{array}{ll}
        1 & \textrm{if }
        \E_i^{l+1}\subseteq \E_j^l \\
        0 & \textrm{otherwise.}
\end{array} \right.\]

Let $F$ be a finite set and $i_0\in F$. Then we denote by $e_{i_0}$
the element in $\Z^F$ for which
\[e_{i_0}(i)=\left\{ \begin{array}{ll} 1 & \textrm{if } i=i_0\\
0 & \textrm{otherwise.}
\end{array} \right. \]
Let $0\le k \le l$. Then we have that $x\sim_l y\implies
\Past_k(x)=\Past_k(y)$. We can therefore for $i\in\{1,2,\dots,m(l)\}$
define $\Past_k(E_i^l)$ to be $\Past_k(x)$ for some $x\in E_i^l$. Let
$M_k^l$ be defined by
\begin{equation*}
  M_k^l=\bigl\{i\in\{1,2\dots ,m(l)\}\mid
\Past_k(\E_i^l)\ne \emptyset\bigr\}.
\end{equation*}
Notice that if $\OSS$ is of the form $\OSS[{\TSS}]$ for some two-sided
shift space $\TSS$ (this is equivalent to $\osh(\OSS)=\OSS$), then
$M_k^l=\{1,2,\dots,m(l)\}$ for all $0\le k \le l$.

If $j\in M_k^l$ and $I_l(i,j)=1$, then $i\in M_k^{l+1}$, so there exists a
positive linear map from $\Z^{M_k^l}$ to $\Z^{M_k^{l+1}}$ given by
\[e_j\mapsto \sum_{i\in M_k^{l+1}}I_l(i,j)e_i.\]
We denotes this map by $I_k^l$.

For a subset $\E$ of $\OSS$ and a $u\in \alwords$, let
$u\E=\{u x\in \OSS \mid x\in \E\}$.
For each $l\in \No,\ j\in \{1,2,\dots ,m(l)\},\ i\in \{1,2,\dots
,m(l+1)\}$ and $a\in \al$, let
\[ A_l(i,j,a)=\left\{ \begin{array}{ll}
        1 & \textrm{if }
        \emptyset \ne a\E_i^{l+1}\subseteq \E_j^l \\
        0 & \textrm{otherwise.}
\end{array} \right.\]

Let $0\le k\le l$. If $j\in M_k^l$ and if there exists an $a\in \al$ such
that $A_l(i,j,a)=1$, then $i\in M_{k+1}^{l+1}$. Thus
there exists a positive linear map from $\Z^{M_k^l}$ to
$\Z^{M_{k+1}^{l+1}}$ given by
\[e_j\mapsto \sum_{i\in M_{k+1}^{l+1}}\sum_{a\in \al}A_l(i,j,a)e_i.\]
We denote this map by $A_k^l$.

\begin{lemma} \label{kisomt}
Let $0\le k \le l$. Then the following
diagram commutes:
\begin{displaymath}
\xymatrix{ \Z^{M_k^l} \ar[r]^{I_k^l} \ar[d]_{A_k^l} & \Z^{M_k^{l+1}} \ar[d]^{A_k^{l+1}} \\
        \Z^{M_{k+1}^{l+1}} \ar[r]^{I_{k+1}^{l+1}} &
        \Z^{M_{k+1}^{l+2}}. }
\end{displaymath}
\end{lemma}

\begin{proof}
  Let $j\in M_k^l$, $h\in M_{k+1}^{l+2}$ and $a\in \al$. If $\emptyset
  \ne a\E_h^{l+2}\subseteq \E_j^l$, then there exists exactly one
  $i\in M_k^{l+1}$ such that $\E_i^{l+1}\subseteq \E_j^l$ and
  $\emptyset \ne a\E_h^{l+2}\subseteq \E_i^{l+1}$; and there exists
  exactly one $i'\in M_{k+1}^{l+1}$ such that $\E_h^{l+2}\subseteq
  \E_{i'}^{l+1}$ and $\emptyset \ne a\E_{i'}^{l+1}\subseteq \E_j^l$.
  If $a\E_h^{l+2}=\emptyset $ or $a\E_h^{l+2}\nsubseteq \E_j^l$,
  then there does not exists an $i\in M_k^{l+1}$ such that
  $\E_i^{l+1}\subseteq \E_j^l$ and $\emptyset \ne a\E_h^{l+2}\subseteq
  \E_i^{l+1}$; and there does not exists an $i'\in M_{k+1}^{l+1}$ such
  that $\E_h^{l+2}\subseteq \E_{i'}^{l+1}$ and $\emptyset \ne
  a\E_{i'}^{l+1}\subseteq \E_j^l$. Hence we have 
  \[\sum_{i\in M_k^{l+1}}A_{l+1}(h,i,a)I_l(i,j)=\sum_{i\in
    M_{k+1}^{l+1}}I_{l+1}(h,i)A_l(i,j,a) .\]
  It follows from this that
  \begin{eqnarray*}
    A_k^{l+1}(I_k^l(e_j))&=& A_k^{l+1}\left(\sum_{i\in
        M_k^{l+1}}I_l(i,j)e_i\right)\\
    &=& \sum_{h\in M_{k+1}^{l+2}}\sum_{a\in
      \al}A_{l+1}(h,i,a)\sum_{i\in M_k^{l+1}}I_l(i,j)e_h\\
    &=& \sum_{h\in M_{k+1}^{l+2}}\sum_{i\in M_{k+1}^{l+1}}\sum_{a\in
      \al}I_{l+1}(h,i)A_l(i,j,a)e_h\\
    &=& I_{k+1}^{l+1}\left(\sum_{i\in M_{k+1}^{l+1}}\sum_{a\in
        \al}A_l(i,j,a)e_i\right) \\
    &=& I_{k+1}^{l+1}(A_k^l(e_j))
  \end{eqnarray*}
  for every $j\in M_k^l$.
  Thus the diagram commutes.
\end{proof}

For $k\in \No$, the inductive limit
$\limm(\Z^{M_k^l},(\Z^+)^{M_k^l},I_k^l)$ will be denoted by
$(\Z_{\OSS_k},\Z_{\OSS_k}^+)$. It follows from Lemma \ref{kisomt}
that the family $\{A_k^l\}_{l\ge k}$ induces a positive, linear map $A_k$ from
$\Z_{\OSS_k}$ to $\Z_{\OSS_{k+1}}$.

Let $0\le k<l$. Denote by $\delta_k^l$ the linear map from $\Z^{M_k^l}$ to
$\Z^{M_{k+1}^l}$ given by
\begin{equation*}
  e_j\mapsto
  \begin{cases}
    e_j&\text{if }j\in M_{k+1}^l,\\
    0&\text{if }j\notin M_{k+1}^l,
  \end{cases}
\end{equation*}
for $j\in M_k^l$. It is easy to check that the following diagram
\begin{displaymath}
\xymatrix{ \Z^{M_k^l}
  \ar[r]^{\delta_k^l} \ar[d]_{I_k^l} & \Z^{M_{k+1}^l} \ar[d]^{I_{k+1}^l} \\
\Z^{M_k^{l+1}} \ar[r]^{\delta_k^{l+1}} & \Z^{M_{k+1}^{l+1}} }
\end{displaymath}
commutes.

Thus the family $\{\delta_k^l\}_{l\ge k}$ induces a positive, linear map from
$\Z_{\OSS_k}$ to $\Z_{\OSS_{k+1}}$ which we denote by
$\delta_k$. Since the diagram
\begin{displaymath}
\xymatrix{ \Z^{M_k^l}
  \ar[r]^{\delta_k^l} \ar[d]_{A_k^l} & \Z^{M_{k+1}^l} \ar[d]^{A_{k+1}^l} \\
\Z^{M_{k+1}^{l+1}} \ar[r]^{\delta_{k+1}^{l+1}} & \Z^{M_{k+2}^{l+1}}  }
\end{displaymath}
commutes for every $0\le k<l$, the diagram
\begin{displaymath}
\xymatrix{ \Z_{\OSS_k} \ar[r]^{\delta_k} \ar[d]_{A_k} &
  \Z_{\OSS_{k+1}} \ar[d]^{A_{k+1}} \\
\Z_{\OSS_{k+1}} \ar[r]^{\delta_{k+1}} & \Z_{\OSS_{k+2}}}
\end{displaymath}
commutes.

We denote the inductive limit $\limm(\Z_{\OSS_k},\Z_{\OSS_k}^+,A_k)$
by $(\Delta_{\OSS},\Delta_{\OSS}^+)$. Since the previous diagram
commutes, the family $\{\delta_k\}_{k\in\No}$ induces a positive, linear map from
$\Delta_{\OSS}$ to $\Delta_{\OSS}$ which we denote by
$\delta_{\OSS}$.

\begin{theorem}
Let $\OSS$ be a one-sided shift space. Then there exists an
isomorphism $\phi:K_0(\Falg)\to\Delta_{\OSS}$ which satisfies that
$\phi(K_0^+(\Falg))=\Delta^+_{\OSS}$ and that
$\phi\circ(\lambda_{\OSS})_0=\delta_{\OSS}\circ\phi$. 
\end{theorem}

For every $l\in \No$ denote by $B^l$ the linear map from $\Z^{m(l)}$
to $\Z^{m(l+1)}$ given by
\[e_j\mapsto \sum_{i=1}^{m(l+1)}\left(I_l(i,j)-\sum_{a\in \al}A_l(i,j,a)\right)e_i.\]

\noindent One can easily check that the following diagram commutes for every
$l\in \No$:
\begin{displaymath}
\xymatrix{ \Z^{m(l)}
  \ar[r]^{B^l} \ar[d]_{I_0^l} & \Z^{m(l+1)} \ar[d]^{I_0^{l+1}} \\
\Z^{m(l+1)} \ar[r]^{B^{l+1}} & \Z^{m(l+2)}. }
\end{displaymath}
Hence the family $\{B^l\}_{l\in\No}$ induces a linear map $B$ from
$\Z_{\OSS_0}$ to $\Z_{\OSS_0}$.

\begin{theorem} \label{heks}
Let $\OSS$ be a one-sided shift space. Then
\[K_0(\Oalg)\cong \Z_{\OSS_0}/B\Z_{\OSS_0},\]
and
\[K_1(\Oalg)\cong \ker(B).\]
\end{theorem}

It follows from Theorem \ref{theorem:osc} and the fact that isomorphic
$C^*$-algebras have isomorphic $K$-theory (cf. Section
\ref{sec:app-ktheory}) that if $\OSS[1]$ and
$\OSS[2]$ are conjugated one-sided shift spaces, then
$(K_0(\OSS[1]),K_0^+(\OSS[1]))\cong (K_0(\OSS[2]),K_0^+(\OSS[2]))$,
$K_1(\OSS[1])\cong K_1(\OSS[2])$ and
$(\Delta_{\OSS[1]},\Delta_{\OSS[1]}^+,\delta_{\OSS[1]})\cong
(\Delta_{\OSS[2]},\Delta_{\OSS[2]}^+,\delta_{\OSS[2]})$, and it follows
from Theorem \ref{prop:morita} and Theorem \ref{theorem:flowmorita} and the fact
that Morita equivalent $C^*$-algebras have isomorphic $K$-theory (cf. Section
\ref{sec:app-ktheory}) that if
$\TSS[1]$ and $\TSS[2]$ are two-sided shift spaces which are
conjugate or just flow equivalent, then
$(K_0(\OSS[1]),K_0^+(\OSS[1]))\cong (K_0(\OSS[2]),K_0^+(\OSS[2]))$. 

We will now prove that if $\TSS[1]$ and $\TSS[2]$ are conjugate
two-sided shift spaces, then we also have that
$(\Delta_{\OSS[{\TSS[1]}]},\Delta_{\OSS[{\TSS[1]}]}^+,\delta_{\OSS[{\TSS[1]}]})\cong
(\Delta_{\OSS[{\TSS[2]}]},\Delta_{\OSS[{\TSS[2]}]}^+,\delta_{\OSS[{\TSS[2]}]})$. It
follows from Theorem \ref{theorem:nasu} that is it enough to prove
this for the case where there
exists a bipartite code between $\TSS[1]$ and $\TSS[2]$. So we will
assume that this is the case and use the same notation as in Section
\ref{sec:two-sided-conjugacy}. 

We let, as in Section \ref{sec:two-sided-conjugacy}, $\OSS$ be the
one-sided shift space $\tilde{f}_1(\OSS[{\TSS[1]}])\cup
\tilde{f}_2(\OSS[{\TSS[2]}])$. Since $\osh(\OSS)=\OSS$, we have that
$M_k^l=\{1,2,\dots,m(l)\}$ for $0\le k\le l$. It is not difficult to see that 
if $l\ge 1$, then each $l$-past equivalence class of $\OSS$ is either
a subset of $\tilde{f}_1(\OSS[{\TSS[1]}])$ or a subset of
$\tilde{f}_2(\OSS[{\TSS[2]}])$. For $i\in\{1,2\}$ we let
$\Z_i^{m(l)}=\Z^{J^l_i}$ where $J^l_i=\{j\in\{1,2,\dots,m(l)\}\mid \E_j^l\subseteq
\tilde{f}_i(\OSS[{\TSS[i]}])\}$. We then have that
$\Z^{m(l)}=\Z_1^{m(l)}\oplus \Z_2^{m(l)}$. It is not difficult to
check that we for all $0\le k\le l$ with $l\ge 1$ have that $I_k^l(
\Z_i^{m(l)})\subseteq \Z_i^{m(l+1)}$ and $A_k^l(
\Z_i^{m(l)})\subseteq \Z_{3-i}^{m(l+1)}$. It is also clear that there
for every $l\ge 1$ exists an isomorphism $\kappa_i^l$ from
$\Z^{m_{\OSS[{\TSS[i]}]}(l)}$ to $\Z_i^{m(2l)}$ such that the following
two diagrams commute:
\begin{displaymath}
\xymatrix{\Z^{m_{\OSS[{\TSS[i]}]}(l)}\ar[r]^{\kappa_i^l}
  \ar[d]_{I_k^l} & \Z_i^{m(2l)} \ar[d]^{I_k^{2l+1}\circ I_k^{2l}} \\
        \Z^{m_{\OSS[{\TSS[i]}]}(l+1)} \ar[r]^{\kappa_i^{l+1}} &
        \Z_i^{m(2l+2)} }
\end{displaymath}
\begin{displaymath}
\xymatrix{\Z^{m_{\OSS[{\TSS[i]}]}(l)}\ar[r]^{\kappa_i^l}
  \ar[d]_{A_k^l} & \Z_i^{m(2l)} \ar[d]^{A_{k+1}^{2l+1}\circ A_k^{2l}} \\
        \Z^{m_{\OSS[{\TSS[i]}]}(l+1)} \ar[r]^{\kappa_i^l} &
        \Z_i^{m(2l+2)} }
\end{displaymath}
where, for each $l$, $m_{\OSS[{\TSS[i]}]}(l)$ denotes the number of
$l$-past equivalence classes in $\OSS[{\TSS[i]}]$ and the maps in the
left column are the ones used to compute the $K$-groups associated to
$\OSS[{\TSS[i]}]$, and the maps in the
right column are the ones used to compute the $K$-groups associated
to $\OSS$. Since $\osh(\OSS[{\TSS[1]}])=\OSS[{\TSS[1]}]$,
$\osh(\OSS[{\TSS[2]}])=\OSS[{\TSS[2]}]$, and $\osh(\OSS)=\OSS$, we
have for every $k\ge 0$ that
$\Z_{(\OSS[{\TSS[1]}])_k}=\Z_{(\OSS[{\TSS[1]}])_0}$,
$\Z_{(\OSS[{\TSS[1]}])_k}=\Z_{(\OSS[{\TSS[1]}])_0}$ and $\Z_{\OSS_k}=\Z_{\OSS_0}$.
It follows from the two commuting diagrams above that there exist
injective homomorphisms 
$\kappa_1:\Z_{(\OSS[{\TSS[1]}])_0}\to\Z_{\OSS_0}$ and
$\kappa_2:\Z_{(\OSS[{\TSS[2]}])_0}\to\Z_{\OSS_0}$ such that
$\Z_{\OSS_0}=\kappa_1(\Z_{(\OSS[{\TSS[1]}])_0})\oplus
\kappa_2(\Z_{(\OSS[{\TSS[2]}])_0})$, $A_0(\kappa_1(\Z_{(\OSS[{\TSS[1]}])_0}))=
\kappa_2(\Z_{(\OSS[{\TSS[2]}])_0})$, and
$A_0(\kappa_2(\Z_{(\OSS[{\TSS[2]}])_0}))= 
\kappa_1(\Z_{(\OSS[{\TSS[1]}])_0})$ and such that the following diagram
commutes:
\begin{displaymath}
\xymatrix{\Z_{(\OSS[{\TSS[1]}])_0} \ar[rr]^{A_0} \ar[d]_{\kappa_1} &&
  \Z_{(\OSS[{\TSS[1]}])_0} \ar[d]^{\kappa_1}&\\
  \kappa_1(\Z_{(\OSS[{\TSS[1]}])_0}) \ar[r]^{A_0} &
  \kappa_2(\Z_{(\OSS[{\TSS[2]}])_0}) \ar[r]^{A_0} &
  \kappa_1(\Z_{(\OSS[{\TSS[1]}])_0}) \ar[r]^{A_0} &
  \kappa_2(\Z_{(\OSS[{\TSS[2]}])_0}) \\
  & \Z_{(\OSS[{\TSS[2]}])_0} \ar[rr]^{A_0} \ar[u]^{\kappa_2}&&
  \Z_{(\OSS[{\TSS[2]}])_0} \ar[u]_{\kappa_2}
}
\end{displaymath}
It follows that $(\Delta_{\OSS[{\TSS[1]}]},\Delta_{\OSS[{\TSS[1]}]}^+,\delta_{\OSS[{\TSS[1]}]})\cong
(\Delta_{\OSS[{\TSS[2]}]},\Delta_{\OSS[{\TSS[2]}]}^+,\delta_{\OSS[{\TSS[2]}]})$. Thus
we have:
\begin{theorem} \label{theorem:fd}
  Let $\TSS[1]$ and $\TSS[2]$ be 
  two-sided shift spaces. If $\TSS[1]$ and $\TSS[2]$ are conjugate,
  then we have that 
$(\Delta_{\OSS[{\TSS[1]}]},\Delta_{\OSS[{\TSS[1]}]}^+,\delta_{\OSS[{\TSS[1]}]})\cong
(\Delta_{\OSS[{\TSS[2]}]},\Delta_{\OSS[{\TSS[2]}]}^+,\delta_{\OSS[{\TSS[2]}]})$.
\end{theorem}

It actually follows from Theorem \ref{theorem:fd} that if $\TSS[1]$
and $\TSS[2]$ are conjugate, then $\Falg[{\OSS[{\TSS[1]}]}]$ and
$\Falg[{\OSS[{\TSS[1]}]}]$ are Morita equivalent.

\appendix
\chapter{Appendix}
I will in this section give a (very short) introduction to $C^*$-algebras, Morita equivalence of $C^*$-algebras and $K$-theory for $C^*$-algebras which hopefully will be enough for the reader to understand these notes.

I will not not give any proofs at all. The interested reader is referred to
for example \cite{MR0512360,MR2188261,MR1656031,MR1074574,MR1634408,MR1783408,MR1222415} for more details.

\section{$\cs$-algebras}

\begin{definition}
A \emph{$\cs$-algebra} is an algebra $\Xalg$ over the complex numbers
equipt with a map 
$x\mapsto x^*$ and a norm $\norm{\cdot}$ 
satisfying:
\begin{enumerate}
\item $\Xalg$ is complete with respect to $\norm{\cdot}$,
\item $\norm{xy}\le\norm{x}\norm{y}$ for $x,y\in\Xalg$,
\item $(x^*)^*=x$ for $x\in\Xalg$,
\item $(xy)^*=y^*x^*$ for $x,y\in\Xalg$,
\item $(\lambda x)^*=\overline{\lambda}x^*$ for $\lambda\in\C$ and $x\in\Xalg$,
\item $(x+y)^*=x^*+y^*$ for $x,y\in\Xalg$,
\item $\norm{x^*}=\norm{x}$ for $x\in\Xalg$,
\item $\norm{x^*x}=\norm{x}^2$ for $x\in\Xalg$.
\end{enumerate}
\end{definition}
The map $x\mapsto x^*$ is called an \emph{involution}. A $C^*$-algebra is called \emph{unital}
if it has a algebraic unit (i.e, $\Xalg$ is unital if there exists a $1\in\Xalg$ such that $1x=x1=x$
for all $x\in\Xalg$). All the $C^*$-algebras we will meet in these notes (except here in the appendix)
are unital.

An algebra equipt with a norm satisfying condition (1) and (2) is
called a \emph{Banach algebra}. 
A Banach algebra equipt with an involution satisfying condition
(3)-(7) is called a \emph{Banach $*$-algebra} (or just a
\emph{$B^*$-algebra}). Condition (8) is often 
called \emph{the $C^*$-identity}. Although this condition at first
glance seems to be a mild condition 
it is in fact very strong because it
ties together the algebraic structure of the $C^*$-algebra and its topology. 
One can for example show that if $\Xalg$ is an algebra
equipt with an involution satisfying condition (3)-(6), then there
is at most one norm which makes $\Xalg$ a $C^*$-algebra.

A map $\phi:\Xalg_1\to\Xalg_2$ between $\cs$-algebras is called a
\emph{$*$-homomorphism} if it satisfies
\begin{enumerate}
\item $\phi(ax+by)=a\phi(x)+b\phi(y)$ for $x,y\in\Xalg_1$ and
  $a,b\in\C$,
\item $\phi(xy)=\phi(x)\phi(y)$ for $x,y\in\Xalg$,
\item $\phi(x^*)=(\phi(x))^*$ for $x\in\Xalg$.
\end{enumerate}

A $*$-homomorphism which is invertible is called a
\emph{$*$-isomorphism}, and if there exists a $*$-isomorphism
between two $C^*$-algebras, then their are said to be \emph{isomorphic}.

If $\phi:\Xalg_1\to\Xalg_2$ is a $*$-homomorphism, then
$\norm{\phi(x)}\le \norm{x}$ for all $x\in\Xalg_1$, and $\phi$ is
injective if and only if $\norm{\phi(x)}=\norm{x}$ for all
$x\in\Xalg_1$ (see for example \cite[Theorem 2.1.7]{MR1074574} for a
proof of this).  
Thus a $*$-homomorphism is automatically
continuous, and a $*$-isomorphism is automatically isometric. 
This is another example of how the algebraic structure of
a $C^*$-algebra and its topology are closely related.

\begin{example}
  Let $\hilbert$ be a Hilbert space. Then the algebra $\B(\hilbert)$
  of bounded operators is a $C^*$-algebra where $T^*$
  of an bounded operator $T\in\B(\hilbert)$ is the adjoint of $T$, and
  the norm $\norm{T}$ is the operator norm $\sup\{\norm{T\eta}\mid
  \eta\in\hilbert,\ \norm{\eta}\le 1\}$.
\end{example}

\begin{definition}
A \emph{projection} in a $C^*$-algebra $\Xalg$ is a $p\in\Xalg$
satisfying $p^2=p^*=p$. 
A \emph{partial isometry} is a $s\in\Xalg$ satisfying $ss^*s=s$.
\end{definition}

It is easy to see that if $s$ is a partial isometry, then $ss^*$ (and
$s^*s$) is a projection. 
One can prove (see for example \cite[Theorem 2.3.3]{MR1074574})
that if $s$ is an element of a $C^*$-algebra such that $ss^*$ is a
projection, then 
$s$ is a partial isometry.

\begin{definition}
  A \emph{$C^*$-subalgebra} of a $C^*$-algebra $\Xalg$ is a closed subalgebra
  $\Yalg$ of $\Xalg$ such that $x\in\Yalg\implies x^*\in\Yalg$.
\end{definition}

A $C^*$-subalgebra $\Yalg$ is a $C^*$-algebra in itself with the
operations it inherits from $\Xalg$.
It is a famous theorem by Gelfand and Naimark that every
$C^*$-algebra is isomorphic to some $C^*$-subalgebra of the
$C^*$-algebra of bounded operators on some Hilbert space.

\begin{example}
Let $X$ be a set. The algebra of bounded functions from $X$ to $\C$ is
a $C^*$-algebra where the involution $f^*$ of a $f\in\BL(X)$ is defined
by $f^*(x)=\overline{f(x)}$ for all $x\in X$, and the norm $\norm{f}$
of $f$ is $\sup\{\abs{f(x)}\mid x\in X\}$. Notice that $\BL(X)$ is abelian.

If $X$ is a locally compact Hausdorff space, then the algebra $C_0(X)$
of continuous functions on $X$ vanishing at infinity is a
$C^*$-subalgebra of $\BL(X)$. 
\end{example}

It is a famous theorem by Gelfand that every abelian
$C$-algebra is isomorphic to $C_0(X)$ for some locally compact
Hausdorff space $X$. 

We are going to need (in the proof of Proposition \ref{prop:dalg}) the
following fact which follows from \cite[Theorem 2.1.11]{MR1074574}: 
 
\begin{fact} \label{fact:inv}
Let $\Xalg$ be a unital $C^*$-algebra. If $\Yalg$ is a
$C^*$-subalgebra of $\Xalg$ which contains the 
unit of $\Xalg$, and $y\in\Yalg$ is invertible in $\Xalg$, then its
inverse $y\inv$ belongs to $\Yalg$. 
\end{fact}

When $\Xalg$ is $C^*$-algebra and $X$ is some subset of $\Xalg$, then
there exists a $C^*$-subalgebra $\Yalg$ of 
$\Xalg$ which contains $X$ and which is contained in any other
$C^*$-subalgebra of $\Xalg$ that contains $X$. The $C^*$-subalgebra 
$\Yalg$ is just the intersection of every $C^*$-subalgebra of $\Xalg$
that contains $X$. 
We call $\Yalg$ \emph{the $C^*$-subalgebra of $\Xalg$ generated by $X$}
and denote it by $C^*(X)$.

\section{Morita equivalence} \label{sec:morita}
By an \emph{ideal} of a $C^*$-algebra we mean a closed two-sided
ideal. I.e., an ideal of a $C^*$-algebra $\Xalg$ is 
a closed subset $I$ of $\Xalg$ such that $\lambda a+\gamma b,xa,ax\in
I$ for $a,b\in I$, $\lambda,\gamma\in\C$ and $x\in\Xalg$. An ideal $I$
of a $C^*$-algebra is automatically closed under involution, i.e., if
$x\in I$, then $x^*\in I$. Thus every ideal of a $C^*$-algebra is also
a $C^*$-subalgebra.

A nonzero ideal of a $C^*$-algebra $\Xalg$ is said to be
\emph{essential} if it has nonzero intersection with every other
nonzero ideal of $\Xalg$. 

There exists for every $C^*$-algebra $\Xalg$ a up to isomorphism
unique maximal unital $C^*$-algebra $M(\Xalg)$ which contains $\Xalg$
as an essential ideal. The $C^*$-algebra $M(\Xalg)$ is known as
\emph{the multiplier algebra of $\Xalg$}, cf. \cite[Theorem
3.1.8]{MR1074574} and \cite[Theorem 2.47]{MR1634408}. If $\Xalg$
itself is unital, then $M(\Xalg)=\Xalg$. 

It is easy to check that if $p$ is a projection in the multiplier
algebra $M(\Xalg)$ of a the $C^*$-algebra $\Xalg$, then $p\Xalg
p:=\{pxp\mid x\in\Xalg\}$ is a $C^*$-subalgebra of $\Xalg$. Such a
$C^*$-subalgebra is called a \emph{corner}.
The projection $p$ is said to be \emph{full} and the corner $p\Xalg p$ 
is said to be a \emph{full corner} if there is no proper ideal of
$\Xalg$ which contains $p$. 

Two projections $p,q\in M(\Xalg)$ are said to be \emph{complementary}
if $p+q=1$. If $p$ and $q$ are complementary, then $pq=0$ and thus
$p\Xalg p\cap q\Xalg q=\{0\}$. In this situation, the two corners
$p\Xalg p$ and $q\Xalg$ are also called complementary. 

\emph{Morita equivalence} is an equivalence relations between
$C^*$-algebras. I will not give the definition of Morita equivalence
here, but instead use the following characterization of Morita equivalence.

\begin{theorem}[{Cf. \cite[Theorem 3.19]{MR1634408}}]
Two $C^*$-algebras $\Xalg_1$ and $\Xalg_2$ are Morita equivalent if
and only if there is a $C^*$-algebra $\Xalg$ with complementary full
corners isomorphic to $\Xalg_1$ and $\Xalg_2$, respectively.
\end{theorem}
It follows directly that Morita equivalence is weaker than isomorphism.
It is not difficult to show that if $p\Xalg p$ is a full corner of a
$C^*$-algebra, then $p\Xalg p$ and $\Xalg$ are Morita equivalent. 

\section{$K$-theory for $C^*$-algebras} \label{sec:app-ktheory}
$K$-theory for $C^*$-algebras is a pair of covariant functors $K_0$ and $K_1$
both defined on the category of $C^*$-algebras.
The functor $K_0$ associate to each $C^*$-algebra $\Xalg$ a pair
$(K_0^+(\Xalg),K_0(\Xalg))$ consisting of an abelian group
$K_0(\Xalg)$ and a sub-semigroup $K_0^+(\Xalg)$ of $K_0(\Xalg)$ (i.e.,
$K_0^+(\Xalg)\subseteq K_0(\Xalg)$ 
and $g,h\in K_0^+(\Xalg)\implies g+h\in K_0^+(\Xalg)$), and associate
to each a $*$-homomorphism $\phi:\Xalg_1\to\Xalg_2$ a group
homomorphism $K_0(\phi):K_0(\Xalg_1)\to K_0(\Xalg_2)$ satisfying 
$K_0(\phi)(K_0^+(\Xalg_1))\subseteq K_0^+(\Xalg_2)$. 
The functor $K_1$ associate to each $C$-algebra $\Xalg$ an abelian
group $K_1(\Xalg)$ and to each a $*$-homomorphism
$\phi:\Xalg_1\to\Xalg_2$ a group homomorphism
$K_1(\phi):K_1(\Xalg_1)\to K_1(\Xalg_2)$. 

That $K_0$ and $K_1$ are functors means that
$K_0(\id\Xalg)=\id_{K_0(\Xalg)}$ and $K_1(\id\Xalg)=\id_{K_1(\Xalg)}$
for every $C^*$-algebra $\Xalg$, and that
$K_0(\phi_1\circ\phi_2)=K_0(\phi_1)\circ K_0(\phi_2)$ and
$K_1(\phi_1\circ\phi_2)=K_1(\phi_1)\circ K_1(\phi_2)$ for all
$*$-homomorphisms $\phi_1:\Xalg_1\to\Xalg_2$ and
$\phi_2:\Xalg_2\to\Xalg_3$. Thus if two $C^*$-algebras are isomorphic,
then $K_0(\Xalg_1)$ and $K_0(\Xalg_2)$ are isomorphic as groups, and
so are $K_1(\Xalg_1)$ and $K_1(\Xalg_2)$. In fact, $K_0(\Xalg_1)$ and
$K_0(\Xalg_2)$ are isomorphic by an isomorphism which maps
$K_0^+(\Xalg_1)$ onto $K_0^+(\Xalg_2)$. 

If $p\Xalg p$ is a full corner of a $C^*$-algebra $\Xalg$ and $\iota$
denotes the inclusion of $p\Xalg p$ into $\Xalg$, then $K_0(\iota)$
and $K_1(\iota)$ are both isomorphisms, and the isomorphism
$K_0(\iota)$ maps $K_0^+(p\Xalg p)$ onto $K_0^+(\Xalg)$, see
\cite[Proposition B.3]{MR2102572}. Thus if two $C^*$-algebras are
Morita equivalent, then $K_1(\Xalg_1)$ and $K_1(\Xalg_2)$ are
isomorphic as groups, and $K_0(\Xalg_1)$ and $K_0(\Xalg_2)$ are
isomorphic as groups by an isomorphism which maps $K_0^+(\Xalg_1)$
onto $K_0^+(\Xalg_2)$.

\end{document}